\documentclass{article}
\usepackage{color}

\usepackage{amsfonts,amssymb,amsthm, authblk, enumerate}
\usepackage{epsfig}
\usepackage{amsmath}
\usepackage{scrlfile}
\usepackage{srcltx}
\usepackage{graphics}
\usepackage{rotating}
\usepackage{remreset}
\newcommand{\beqa}{\begin{eqnarray}}
\newcommand{\eeqa}[1]{\label{#1}\end{eqnarray}}
\newcommand{\beq}{\begin{equation}}
\newcommand{\eeq}[1]{\label{#1}\end{equation}}

\def\R{\mathbb{R}}
\def\Z{\mathbb{Z}}

\def\N{\mathbb{N}}

\def\l{\lambda}

\def\phi{\varphi}

\def\O{\Omega}

\def\demifleche{\rightharpoonup}

\def\*fleche{\buildrel *\over\demifleche}

\def\tol2{\buildrel\hbox{$L^2$}\over\longrightarrow}

\def\toto{\leaders\hbox to 5mm{\hfil.\hfil}\hfill}

\newcommand{\dist}{\operatorname{dist}}

\newcommand{\ep}{\varepsilon}

\newcommand{\RR}{\mathbb{R}}

\newtheorem{theorem}{Theorem}[section]

\newtheorem{corollary}[theorem]{Corollary}

\newtheorem{lemma}[theorem]{Lemma}

\newtheorem{remark}[theorem]{Remark}

\newcommand{\e}{\varepsilon}

\begin{document}
\title{\sc Extreme localisation of eigenfunctions to one-dimensional high-contrast periodic problems with a defect}

\author[1]{Mikhail Cherdantsev}
\author[2]{Kirill Cherednichenko}
\author[2]{Shane Cooper}
\affil[1]{Cardiff School of Mathematics, Senghennydd Road, Cardiff, CF24 4AG}
\affil[2]{Department of Mathematical Sciences, University of Bath, Claverton Down, Bath, BA2 7AY, UK}
\maketitle

\begin{abstract}
Following a number of recent studies of resolvent and spectral convergence of non-uniformly elliptic families of differential operators describing the behaviour of periodic composite media with high contrast, we study the corresponding one-dimensional version that includes a ``defect'': an inclusion of fixed size with a given set of material parameters. It is known that the spectrum of the purely periodic case without the defect  and its limit, as the period $\varepsilon$ goes to zero, has a band-gap structure. We consider a sequence of eigenvalues $\l_\e$ that are induced by the defect and converge to a point $\l_0$ located in a gap of the limit spectrum for the periodic case. We show that the corresponding eigenfunctions are ``extremely'' localised to the defect, in the sense that the localisation exponent behaves as $\exp(-\nu/\varepsilon),$ $\nu>0,$ which has not been observed in the existing literature.
As a consequence, we argue that $\l_0$ is an eigenvalue of a certain limit operator defined on the defect only. In two- and three-dimensional configurations, whose one-dimensional cross-sections are described by the setting considered, this implies the existence of propagating waves that are localised to the defect.  We also show that the unperturbed operators are norm-resolvent close to a degenerate operator on the real axis, which is described explicitly.  

{\bf Keywords} High-contrast homogenisation $\cdot$  Wave localisation $\cdot$ Spectrum $\cdot$ Decay estimates

\end{abstract}

\section{Introduction}

The question of whether a macroscopic perturbation of material properties in a periodic medium or structure (periodic composite) induces the existence of a localised solution (bound state) to the time-harmonic version of the equations of motion is of special importance from the physics, engineering and mathematical points of view. Depending on the application context, such a solution can have either an advantageous or undesirable effect on the behaviour of systems containing the related composite medium as a component. For example, in the context of photonic (phononic) crystal fibres, perturbations of this kind have been exploited for the transport of electromagnetic (elastic) energy over large distances with little loss into the surrounding space, see {\it e.g} \cite{Knight, Russell}. In the mathematics literature, proofs of the existence or non-existence of such a localised solution have been carried out using the tools of the classical asymptotic analysis of the governing equations and spectral analysis of operators generated by the governing equations in various ``natural" function spaces. The choice of the concrete class of equations and functions under study is usually motivated by the applications in mind, and several works that have marked the development of the related analytical techniques cover a wide range of operators and their relatively compact perturbations, {\it e.g.} \cite{Weyl}, \cite{Birman1961}, \cite{AADH}, \cite{FiKl}.
 
The present work is a study of localisation properties for a class of composite media that has been the subject of increasing interest in the mathematics and physics literature recently, in view of it relation to the so-called metamaterials,  {\it e.g.} manufactured composites possessing negative refraction properties. It has been shown in \cite{ChCoGu} that the spectrum of a stratified high-contrast composite, represented mathematically by a one-dimensional periodic second-order differential equation, has an infinitely increasing number of gaps (lacunae) opening in the spectrum, in the limit of the small ratio $\varepsilon$ between the period and the overall size of the composite. This analytical feature, analogous to the spectral property of multi-dimensional high-contrast periodic composites shown in \cite{Zhikov2000}, provides a mathematical recipe for the use of such materials in physics context or technologies where the presence of localised modes (generated by defects in the medium) has important practical implications. In the physical context of photonic crystal fibres and within the mathematical setup of multi-dimensional high-contrast media, this link has been studied in \cite{KamSm1}, \cite{Cherdantsev1}, \cite{Cherdantsev2}. In the paper \cite{KamSm1}, a two-scale asymptotics for eigenfunctions of a high-contrast second-order elliptic differential operator with a finite-size perturbation (defect) was derived. It was shown that for eigenvalues $\lambda$ in gaps of the spectrum of the (two-scale) operator representing the leading-order term of this asymptotics, there are sequences of eigenvalues of the finite-period problems that converge to $\lambda$ as $\varepsilon\to0.$ The subsequent works \cite{Cherdantsev1}, \cite{Cherdantsev2} developed a multiscale version of Agmon's approach \cite{Agmon} and proved that the corresponding eigenfunctions of the limit operator decay exponentially fast away from the defect.  An important technical assumption in all these works is that the low-modulus inclusions in the composite have a positive distance to the boundary of the period cell, which is not possible to satisfy in one dimension. 

In the more recent paper \cite{ChCoGu}, a family of non-uniformly elliptic periodic one-dimensional problems with high contrast was studied, which in practically relevant situations corresponds to a stratified composite with alternating layers of homogeneous media with highly contrasting material properties. It was shown that the spectra of the corresponding operators converge, as $\varepsilon\to0$, to the band-gap spectrum of a  two-scale operator described explicitly in terms of the original material parameters. Introducing a finite-size defect $D$ into the setup of \cite{ChCoGu}, one is led to consider the operator
\[
-(a_D^\varepsilon u')',\ \ \ \ a_D^\varepsilon>0,
\]
where $a_d^\varepsilon$ takes values of order one on $D$ and $\varepsilon$-periodic ($\varepsilon>0$) in ${\mathbb R}\setminus D$ with alternating values of order one and $\varepsilon^2.$ As was mentioned in \cite[Section 5.1]{ChCoGu}, a formal analysis suggests that the rate of decay of eigenfunctions localised in the vicinity of the perturbation $D$ 
is ``accelerated exponential'', rather than just exponential as in \cite{Cherdantsev2}, in the sense that the decay exponent increases in absolute value as $\varepsilon\to0.$ The goal of the present work is to provide a rigorous proof of this property, formulated below as Theorem \ref{mainthm}. In view of the above discussion, this new localisation property can be seen as a consequence of two features of the underlying periodic composite: loss of uniform ellipticity (via the presence of soft inclusions in the moderately stiff matrix material) and the geometric condition of the matrix material components being separated by the inclusions. 


In addition to our main result, we formulate (Section \ref{sec:limspectrum}) a new characterisation of the limit spectrum for the unperturbed family of problems in the whole space discussed in \cite{ChCoGu} and strengthen (Section \ref{sec:resolvent}) the result of \cite{ChCoGu} by proving an order-sharp norm-resolvent convergence estimate for this family (Theorem \ref{estlemma}). In particular, this new estimate implies order-sharp uniform convergence estimates, as $\varepsilon\to0,$ for the related family of parabolic problems, via the norm-resolvent convergence of the corresponding operator semigroups.


\section{Problem formulation and main results}
\label{sec:probform}
For $\varepsilon, h\in(0,1),$
we introduce the sets
\[
\Omega_0^\e: = \bigcup_{z \in \mathbb Z}(\e z , \e z +\e h )
 ,\qquad \text{and} \qquad  \Omega_1^\e : = \bigcup_{z \in \mathbb Z}(\e z + \e h , \e z +\e )
   = \mathbb R  \setminus \overline{\Omega_0^\e},
\]
and denote $Y_0:=(0,h)$, $Y_1:=(h,1)$, $Y:=(0,1)$. We define the $\e$-periodic functions 
\begin{equation}
\label{upcoefs}
\begin{aligned}
 a^\e(x):=\left\{\begin{array}{lll}
                     \e^2 a_0(\tfrac{x}{\e}), & x \in \Omega_0^\e,
                    \\[0.3em]
                    a_1(\tfrac{x}{\e}), & x \in \Omega_1^\e,
                   \end{array}
\right.  & \hspace{.5cm} & \rho^\ep(x) = \rho(\tfrac{x}{\ep}), &\hspace{.5cm} &  \rho(y):=\left\{\begin{array}{lll}
\rho_0(y), & \text{$y\in Y_0$},
\\[0.3em]
\rho_1(y), & \text{$y\in Y_1$},
\end{array}
\right. 
\end{aligned}
\end{equation}
for  $a_j, a_j^{-1}, \rho_j, \rho_j^{-1} \in L^\infty(Y_j)$, $j=1,2$, periodic with period $1$. It is convenient to set $a_0\equiv 0$ on $Y_1$ and $a_1\equiv 0$ and $Y_0$, thus we can write, for example, $a^\e(x) = \e^2 a_0(x/\e) + a_1(x/\e)$.

For a positive Lebesgue-measurable function $w$ on a Borel set $B\subset{\mathbb R},$ such that $w, w^{-1} \in L^\infty(B)$, we employ the notation  $L^2_{w}(B)$ to indicate that the space $L^2(B)$ is equipped with the inner product 
$$
( \cdot, \cdot)_{_w} : =  \int_{B} w | \cdot |^2.
$$
For a bilinear form $\beta : H^1(\RR) \times H^1(\RR) \rightarrow \RR$,  the  (self-adjoint) operator $A$ associated to $\beta$ is densely defined in $L^2_w(\RR)$ by the action $A u = f,$ where for a given $u \in H^1(\RR)$, the function $f \in L^2(\RR)$ is the solution to the integral identity
$$
\beta(u, v ) = \int_{\RR}w f\overline{v} \qquad \forall v  \in H^1(\RR).
$$

For the bilinear form
\begin{equation*}
\begin{aligned}
\beta^\e (u,v) : = \int_\RR a^\e u'\overline{v'},\qquad u,v\in H^1(\RR),
\end{aligned}
\end{equation*}
we consider $A^\e,$ the operator defined in $L^2_{\rho^\e}(\RR)$ and associated to $\beta^\ep$. The spectrum $\sigma(A^\e)$ of $A^\e$ is absolutely continuous and, by introducing the rescaled Floquet-Bloch transform ${\mathcal U}_\varepsilon$, see \eqref{rsgelf}, we note that  $\sigma(A^\ep)$ admits the representation
 $$
 \sigma(A^\ep) = \bigcup_{\theta \in [0,2\pi)} \sigma( A^\ep_\theta),
 $$
 where $\sigma(A_\theta^\e)$ is the spectrum of the $L^2_{\rho}(Y)$ densely-defined self-adjoint operator $A^\ep_\theta$ associated to the form
 $$
\beta^\e_\theta (u,v) : = \int_Y\bigl(a_0 + \ep^{-2} a_1\bigr) u'\overline{v'},
 $$
acting in the space $H^1_\theta(Y)$ of functions $u\in H^1(Y)$ that are $\theta$-quasiperiodic, {\it i.e.} such that $u(y)=\exp({\rm i}\theta y)v(y),$ $y\in Y,$ for some $1$-periodic function $v\in H^1(Y)$. For each $\varepsilon,$ $\theta,$ the operator $A_\theta^\ep$ has compact resolvent and consequently its spectrum $\sigma(A^\ep_\theta)$ is discrete.

Consider the space 
\[
V_\theta:=\bigl\{u\in H^1_\theta(Y):\,u'=0 \mbox{ on }  Y_1\bigr\}
\] 
and its closure in $L^2_\rho(Y),$ which we denote by $\overline{V_\theta}.$ We introduce the densely defined operators $A_\theta : \overline{V_\theta} \rightarrow L^2_\rho(Y)$ given by $A_\theta u = f$ for $u,f$ such that
\begin{equation}
\label{Atheta}
 \int_{Y_0} a_0 u' \overline{v'} = \int_Y \rho f \overline{v} \qquad \forall v \in V_\theta.
\end{equation}
For each $\theta,$ the operator $A_\theta$ has compact resolvent and so $\sigma(A_\theta)$ is discrete. 
 In a recent work \cite{ChCoGu}, see Section \ref{sec:resolvent}, the spectrum $\sigma(A^\e)$ was shown to converge in the Hausdorff sense to  the union of  the spectra  of the operators $A_\theta,$ {\it i.e.} 
 \begin{equation}
 \lim_{\e \rightarrow 0} \sigma(A^\e) = {\bigcup_{\theta \in [0,2\pi)} \sigma( A_\theta)}.
 \label{limspectrum}
 \end{equation} 
 \begin{remark}\label{rem:2.1}
 	$\bigcup_{\theta \in [0,2\pi)} \sigma( A_\theta)$ can be seen as the spectrum of a certain operator $A^0$ unitary equivalent to the direct integral of operators $\int_\theta A_\theta$, see Appendix \ref{rem:limoperator} for the details.
 \end{remark}
 In Section \ref{sec:resolvent}, we construct infinite-order asymptotics (as $\varepsilon\to0$) for the resolvents of $A_{\theta}^\e$, uniform in $\theta,$ with respect to the $H^1$ norm and, in particular, prove the following refinement of the result established in \cite{ChCoGu}:
 \begin{theorem}
 \label{estlemma}
  The  operator $A^\e_\theta$ norm-resolvent converges to $A_\theta$, uniformly in $\theta$,  at the rate $\ep^2$. More precisely, there exists a constant $C>0$ such that
  $$
  \bigl\Vert( A^\e_\theta + 1)^{-1}f - (A_\theta + 1)^{-1} f\bigr\Vert_{L^2_\rho(Y)} \le C  {\e}^2 || f ||_{L^2_\rho(Y)}\ \  \qquad  \forall \theta \in [0,2\pi).
  $$
 \end{theorem}
  Consequently, since the spectra $\sigma(A^\e_\theta)$ and $\sigma(A_\theta)$ are discrete, we have the following result: for each $n \in \mathbb{N}$ there exists a constant $c_n >0$ such that
  $$
\bigl| \lambda^\ep_{n}(\theta) - \lambda_{ n }(\theta) \bigr| \le c_n \ep^2 \ \ \qquad \forall \theta \in [0,2\pi),\ \ \ \varepsilon \in (0,1).
   $$
 Here, $\{ \lambda^\e_{n}(\theta) \}_{n \in \mathbb{N}}$, $\{ \lambda_{n}(\theta) \}_{n \in \mathbb{N}}$ are the eigenvalue sequences of $A^\e_\theta$, $A_\theta,$ respectively, labelled in the increasing order\footnote{Notice that all the eigenvalues are simple due to the 1-dimensional nature of the corresponding problem.}. From this theorem  it follows that for sufficiently small $\ep$, the spectrum $\sigma(A^\e)$ has gaps if the set ${\bigcup_\theta \sigma( A_\theta)}$ contains gaps. 
 In Section \ref{sec:limspectrum} we demonstrate that this set contains infinitely many gaps. Furthermore, we demonstrate that $\lambda \in {\bigcup_\theta \sigma( A_\theta)}$ if and only if the inequality
 $$
 \Big|  v_1(h) + (a_0 v_2')(h) - \lambda v_2(h)\int_{Y_1} \rho_1 \Big| \le 2
 $$
 holds. Here $v_1$ and $v_2$ are the ($\lambda$-dependent) solutions of 
 \begin{equation*}
 - (a_0 v_j')' = \lambda \rho_0 v_j  \ \ \ \text{on $Y_0$},\ \ \ \ \ \ \ \ \ \ j=1,2,\\
 \end{equation*}
 subject to the conditions
 \begin{equation*}
 \left(\begin{array}{cc} v_1(0) & v_2(0)\\[0.4em] (a_0 v_1')(0) & (a_0 v_2')(0)\end{array}\right)
 = \left(\begin{array}{cc} 1 & 0\\[0.4em] 0 & 1\end{array}\right).
 \end{equation*}
\begin{remark}\label{rem2.3}
Note that any solution $u$ of  $- (a_0 u')' = \lambda \rho_0 u$ is absolutely continuous and so is its co-derivative $a_0 u'$. Hence, their value at any point $y$ is well defined (unlike the value of $a_0$ or $u'$ in general). This explains the use of notation $(a_0 v_j')(y)$, which we will hold to throughout the paper.
\end{remark}
 
 
Next, we introduce $d_-,d_+\in{\mathbb R}$ and on the set $D=(d_-,d_+)$ replace the coefficients \eqref{upcoefs} by some uniformly positive and bounded functions $a_D$, $\rho_D,$ namely we consider
\begin{equation*}
\begin{aligned}
 a_D^\e(x):=\left\{\begin{array}{lll}
                    a_D(x), & x\in D,
                    \\[0.3em]
                    a_1(\tfrac{x}{\e}), & x \in \Omega_1^\e \backslash D,
                    \\[0.3em]
                    \e^2 a_0(\tfrac{x}{\e}), & x \in \Omega_0^\e \backslash D, 
                   \end{array}
\right.  & \hspace{1cm} &  \rho_D^\e(x):=\left\{\begin{array}{lll}
\rho_D(x), & x\in D,
\\[0.3em]
\rho_1(\tfrac{x}{\e}), & x \in \Omega_1^\e \backslash D,
\\[0.3em]
\rho_0(\tfrac{x}{\e}), & x \in \Omega_0^\e \backslash D.
\end{array}
\right. 
\end{aligned}
\end{equation*}
We shall study the spectrum of the operator $A_D^\e$  defined in $L^2_{\rho_D^\e}(\RR)$  and associated to the form
\begin{equation}\label{weakdefectform}
\begin{aligned}
\beta_D^\e(u,v) : = \int_\RR a_D^\e u'\overline{v'}, \qquad u, v\in H^1(\R).
\end{aligned}
\end{equation}
As this operator arises from a compact  perturbation of the coefficients of $A^\ep$, it is well-known, see {\it e.g.} \cite{FiKl}, that the essential spectra of $A^\e_D$ and $A^\e$ coincide. For eigenvalues situated, for small values of $\varepsilon,$ in the gaps of the essential spectrum of $A^\e_D$ (equivalently, in the gaps of the essential spectrum of $\sigma(A^\e)$), we expect the eigenfunctions to be localised around the defect, and therefore we are interested in the analysis of  eigenfunctions of $A^\e_D$ corresponding to eigenvalues that are located in the gaps of the limit spectrum ${\bigcup_\theta \sigma(A_\theta)}$. We show that for the sequence of point spectra $\sigma_{\rm p}(A^\e_D)$ of the operators $A^\e_D,$ the set of accumulation points as $\varepsilon\to0$ that are located in the gaps 
of the limit spectrum $\lim_{\varepsilon\to0}\sigma(A^\varepsilon)$ is given by the intersection of the set ${\mathbb R}\backslash {\bigcup_\theta \sigma(A_\theta)}$ with the spectrum of the operator $A_{{\rm N},D}$ defined in $L^2_{\rho_{_D}}(D)$ and associated to the form 
\begin{equation}\label{A_N,D}
\begin{aligned}
\beta_{{\rm N},D}(u,v) : = \int_D a_D u'\overline{v'},\qquad u, v\in H^1(D).
\end{aligned}
\end{equation}
The functions from the domain of $A_{{\rm N},D}$ satisfy the Neumann condition on the boundary of $D.$ 

Conversely, if we choose the defect $D$ so that the spectrum $\sigma(A_{{\rm N},D})$ has a non-empty intersection with $\mathbb{R} \backslash {\bigcup_\theta \sigma(A_\theta)}$, then for sufficiently small $\e$ the operator $A^\e_D$ has non-empty point spectrum. Notice that we can always choose $a_D$, $\rho_D$, $d_-$ and $d_+$ such that this is true. Moreover, we demonstrate that for eigenvalue sequences that converge to a point in $\mathbb{R} \backslash {\bigcup_\theta \sigma(A_\theta)}$ the corresponding eigenfunctions are localised to a small neighbourhood of the defect.  Namely, the eigenfunctions $u_\e$  exhibit accelerated exponential decay  outside the defect in the sense that the function 
$\exp\bigl({\rm dist}(x, D)\nu/\e\bigr)u_\e(x),$ $x\in{\mathbb R},$ is an element of $L^2(\mathbb{R}\backslash D)$ for sufficiently small $\e,$ where the value $\nu$ is determined by the distance of the limit point of $\lambda_\e$ to the set ${\bigcup_\theta \sigma(A_\theta)}$.
These results are collated in the following theorem, which we prove in Sections \ref{sec:formalasymp}, \ref{sec:uppersemicont}, \ref{sec:decay}.
\begin{theorem}
\label{mainthm} $\ $

\begin{enumerate}
\item{For every  $\lambda_0 \in\sigma(A_{{\rm N},D})\backslash\bigl({\bigcup_\theta \sigma(A_\theta)}\bigr)$ (which is always simple) there exist a unique\footnote{By this we mean asymptotically unique, {\it i.e.} if there are two sequences $\l_\e$ and $\l_\e'$ converging to $\l_0$ then necessary $\l_\e = \l_\e'$ for small enough $\e$.} sequence of simple eigenvalues $\l_\e$ of $A^\e_D$  converging to $\l_0$ and constants $C_1,C_2>0$ such that 
\begin{equation}\label{2.6}
	\begin{aligned}
	&	|\lambda_\e - \l_0| \le C_1 \ep, 
	\\
& 	\|u_\e - u_0\|_{L^2(D)} \leq C_1 \e^{1/2}, \\
	&	\|u_\e \|_{L^2(\R \backslash D)} \leq C_2 \e^{1/2},
	\end{aligned}
\end{equation}
	where $u_0$ and $u_\e$ are the eigenfunctions of  $\l_0$ and $\l_\e$ respectively.
}
\item{ Conversely,  for any sequence $\lambda_\e \in \sigma_p(A^\e_D)$
such that $\lambda_\e \rightarrow \lambda_0\notin\lim\limits_{\e \rightarrow 0} \sigma(A^\e)= { \bigcup_\theta \sigma(A_\theta)}$, one has 
$\lambda_0 \in \sigma(A_{{\rm N},D})$.}
\item{Furthermore, the  $L^2(\mathbb{R})$-normalised eigenfunctions $u_\e$ of $A_D^\varepsilon$ corresponding to the eigenvalues 
$\lambda_\e$ are localised to the defect in the following sense: for $\nu >0$, let  $g_{\nu/\e}$ denote the exponentially growing function
			\begin{equation}\label{2.4}
			g_{\nu/\e}(x):= \left\{\begin{array}{lll}
			1, & x\in D,
			\\[0.3em]
			\exp\big( \frac{\nu}{\e} \dist(x,D)\big), & x\in \R\setminus D, 
			\end{array}\right.
			\end{equation}
and take	$\mu_1$ to be the smallest by the absolute value root of the quadratic function
		        \begin{equation}
			q(\mu):=\mu^2 - \biggl( v_1(h) + (a_0v_2')(h) - \lambda_0 v_2(h) \int_{Y_1} \rho_1 \biggr)\mu + 1.
			\label{qdef}
			\end{equation}	
Then,  for sufficiently small values of $\e$  
the function $g_{\nu/\e}\, u_\e$ is an element of $L^2(\mathbb R)$ for all  $\nu < |\ln | \mu_1| |$.

}
\end{enumerate}
\end{theorem}

\section{The limit spectrum of the unperturbed operator}
\label{sec:limspectrum}
Here we quantitatively characterise the spectrum ({\it cf.} (\ref{limspectrum})) 
$$
{\bigcup_{\theta \in [0,2\pi)} \sigma( A_\theta)}
$$
and establish criteria for the existence of spectral gaps. To this end we consider the eigenvalue problem: find $\lambda \in [0,\infty)$ and $u \in V_\theta=\bigl\{v\in H^1_\theta(Y):\,v'\equiv0 \mbox{ on }  Y_1\bigr\}$ such that
\begin{equation}\label{7.76}
\int\limits_0^h a_0 u' \overline{v'} 
= \l\int\limits_0^1 \rho u \overline{v} 
\qquad \forall v\in V_\theta.
\end{equation}
By taking test functions $v\in C^\infty_0(Y_0)$ we deduce that $u \vert_{Y_0}$ is a weak solution to the equation
\begin{equation}
-(a_0 u')'=\lambda\rho_0 u
\label{strongform}
\end{equation}
on $Y_0.$
For $L^\infty$-functions $a_0$ and $\rho$ the equation (\ref{strongform}) holds pointwise almost everywhere and by integrating by parts in \eqref{7.76} we deduce that
$$
(a_0u')(h^-)\,\overline{v(h)} - (a_0u')(0^+)\,\overline{v(0)}  = \lambda \int_{Y_1} \rho_1 u \overline{ v} \qquad \forall v \in V_\theta.
$$ 
Here $ f(z^+): = \lim\limits_{x \searrow z}f(x),$ and $f(z^-) := \lim\limits_{x \nearrow z}f (x) $ for a function $f,$ whenever the corresponding limit exists. Since any element $v\in V_\theta$ satisfies $v(y)= e^{{\rm i} \theta } v(0),$ $y\in Y_1,$  the above observations imply that $u$ satisfies \eqref{7.76} if, and only if, $w = u\vert_{Y_0} \in H^1(Y_0)$ is a weak solution of the problem
\begin{equation}
\label{equivlim}
\left\{
\begin{aligned}
& - (a_0 u')' = \lambda \rho_0 u \quad  \quad\text{in $Y_0$},
 \\[0.3em]
& u(h)  = e^{{\rm i} \theta } u(0),   
\\
& e^{-{\rm i} \theta }(a_0u')(h^-) -(a_0 u')(0^+)=\lambda u(0) \int_{Y_1} \rho_1.
\end{aligned}
\right.
\end{equation}

We now describe the solutions to  \eqref{equivlim},  equivalently \eqref{7.76}.
\subsection{Representation via a fundamental system}
\label{sec:transfer}
Due to the existence and uniqueness theorem for linear first order systems with locally integrable coefficients, see e.g. \cite{Zettl}, there exist a fundamental system of solutions $v_1$ and $v_2$ to the equation $- (a_0 u')' = \lambda \rho_0 u$  such that
\begin{equation}\label{7.80}
\left(\begin{array}{cc} v_1(0) & v_2(0) \\[0.4em] (a_0 v_1')(0) & (a_0 v_2')(0)\end{array}\right)
= \left(\begin{array}{cc} 1 & 0\\[0.4em] 0 & 1\end{array}\right),
\end{equation} 
cf. Remark \ref{rem2.3}.
The Wronskian of the system is constant:
\begin{equation}\label{7.81}
v_1(y)(a_0 v_2')(y) - v_2(y)(a_0 v_1')(y)=1,\,\, y\in Y_0,
\end{equation}
and any solution $u$ to \eqref{equivlim} is of the form
$
u = c_1 v_1 + c_2 v_2
$
for some $c_1, c_2\in{\mathbb C}.$ 

The substitution of the above representation for $u$ in terms of $v_1,$ $v_2$ into the second and third equations of \eqref{equivlim} leads to the system
\begin{equation}
\left( \begin{matrix}
v_1(h) -  e^{{\rm i} \theta }  & & v_2(h) \\[5pt]
(a_0 v_1')(h) -  e^{{\rm i} \theta } \lambda \int_{Y_1} \rho_1 & & (a_0 v_2')(h) - e^{{\rm i} \theta }
\end{matrix} \right) \left( \begin{matrix}
c_1 \\ c_2
\end{matrix} \right) = \left( \begin{matrix}
0 \\ 0
\end{matrix} \right).
\label{csystem}
\end{equation}
For the existence of a non-trivial solution $(c_1,c_2)$ to (\ref{csystem}), and therefore non-trivial $u$ in \eqref{equivlim}, the value $\lambda$ must necessarily solve the equation
$$
2\cos ( \theta ) = v_1(h) + (a_0 v_2')(h) - \lambda v_2(h)\int_{Y_1} \rho_1.
$$
Hence, the set ({\it cf.} (\ref{limspectrum}))
$$
{\bigcup_{\theta\in [0,2\pi)} \sigma( A_\theta)}
$$
consists of all non-negative $\lambda$ such that the following inequality holds:
\begin{equation}
 \bigg|  v_1(h) + (a_0 v_2')(h) - \lambda v_2(h)\int_{Y_1} \rho_1 \bigg| \le 2.
 \label{specineq}
\end{equation}
As noted above, the functions $v_1, v_2$ depend on the spectral parameter $\l$. We demonstrate the implications of this dependence for the structure of the limit spectrum through the following simple example. Assume that $a_0,\rho_0$ and $\rho_1$ are equal to unity on their support, then $v_1 = \cos (\sqrt{\l} y)$, $v_2 = \frac{1}{\sqrt{\l}}\sin (\sqrt{\l} y)$, and the limit spectrum is given by
\begin{equation*}
\bigg|  2 \cos (\sqrt{\l} h)  - \sqrt{\lambda} \sin (\sqrt{\l} h) (1-h) \bigg| \le 2.
\end{equation*}
In particular we see that the bands of the limit spectrum become very narrow as $\l \to \infty$.

\subsection{Representation via a spectral decomposition}

Consider the operator $\tilde A_\theta$ defined on $L^2_{\rho_0}(Y_0)$ and associated to the form 
$$
\tilde\beta_\theta(u,v) : = \int_{Y_0} a_0 u' \overline{v'}, \qquad u, v \in H^1_\theta(Y_0),
$$
in the sense of procedure described in Section \ref{sec:probform}.
By virtue of the fact that the operator $\tilde A_\theta$ has compact resolvent, its $L^2_{\rho_0}(Y_0)$-orthonormal sequence of eigenfunctions $\{ \Phi^{(n)}_\theta\}_{n \in \N}$ is complete in the space $L^2_{\rho_0}(Y)$. We denote by $\mu_n(\theta)$, $n\in \N$, the eigenvalues of $\Phi^{(n)}_\theta \in H^1_\theta(Y_0)$: 
\begin{equation}
\label{examplee1}
\begin{aligned}
\int_{Y_0} a_0\bigl(\Phi^{(n)}_\theta\bigr)'\overline{v'} = \mu_n(\theta) \int_{Y_0} \rho_0 \Phi^{(n)}_\theta \overline{v} & \qquad & \forall v \in H^1_\theta(Y_0).
\end{aligned}
\end{equation}
Multiplying the first equation in (\ref{equivlim}) by $\overline{\Phi^{(n)}_\theta}$ and integrating by parts we have
\begin{flalign*}
\l \int_{Y_0} \rho_0 u \overline{\Phi^{(n)}_\theta}= -\int_{Y_0} (a_0u')' \overline{\Phi^{(n)}_\theta}&= -\big((a_0 u')(h^-)\overline{\Phi^{(n)}_\theta(h)}-
(a_0 u')(0^+)\overline{\Phi^{(n)}_\theta(0)}\big) + \int_{Y_0} a_0u'\,\overline{\big({\Phi^{(n)}_\theta}\big)'} 
\\
& =-\left( e^{-{\rm i} \theta } (a_0u')(h^-)  - (a_0u')(0^+) \right) \overline{\Phi^{(n)}_\theta(0)}   + \mu_{n}(\theta) \int_{Y_0} \rho_0 u\,\overline{\Phi^{(n)}_\theta}.
\end{flalign*}
The third equation in  (\ref{equivlim}) implies
$$
\bigl(\mu_{n}(\theta)-\l\bigr) \int_{Y_0} \rho_0 u\,\overline{\Phi^{(n)}_\theta} = \lambda u(0) \overline{\Phi^{(n)}_\theta(0)} \int_{Y_1} \rho_1
$$
 Therefore, upon performing a spectral decomposition of $u$ in terms of $\Phi^{(n)}_\theta,$ {\it i.e.} setting
$$
\begin{aligned}u = \sum_{n \in \N} \zeta_n \Phi^{(n)}_\theta, & \qquad& \zeta_n = \int_{Y_0} \rho_0 u\,\overline{\Phi^{(n)}_\theta}, 
\end{aligned}
$$
we see that
$$
\zeta_n = \frac{\lambda}{\mu_{n}(\theta)- \lambda} u(0) \overline{\Phi^{(n)}_\theta(0)}\int_{Y_1} \rho_1,\ \ \ \ n\in{\mathbb N}.
$$
In particular, one has $u(0) = \sum_{n \in \N} \zeta_n \Phi^{(n)}_\theta(0)$. Thus we arrive at the following alternative description of the limit spectrum: $\lambda\in{\bigcup_\theta \sigma(A_\theta)}$ if and only if there exist $\theta \in  [0,2\pi)$ such that 
$$
 \sum_{n \in \N} \frac{\lambda}{\mu_{n}(\theta)- \lambda}\bigl| \Phi^{(n)}_\theta(0)\bigr|^2  =\left( \int_{Y_1} \rho_1 \right)^{-1}.
$$
\section{Asymptotics of the defect eigenvalue problem}
\label{sec:formalasymp}
Suppose $\lambda_\ep$, $u_\ep$ is an eigenvalue-eigenfunction pair for the defect problem, that is
\begin{equation}
\label{formale1}
\begin{aligned}
-(a_D^\e u_\e')' & = \lambda_\e \rho_D^\e u_\e \quad \text{on $\mathbb{R}$}, \\ 
\end{aligned} 
\end{equation}
where $u_\ep$ is continuous, subject to the interface conditions
\begin{equation}
\label{formale2}
 a_D u'_\e \big|_{ D} = a^\e_D u'_\e \big|_{\RR \backslash D}  \quad {\rm on}\ \ \{d_-, d_+\}.
\end{equation}
and
\begin{equation}
\label{formale3}
 a_1 u'_\e \big|_{\Omega^\e_1 \backslash D} = \ep^2 a_0 u'_\e \big|_{\Omega^\e_0 \backslash D} \quad {\rm on}\ \ \bigl\{x\in\RR \backslash D : x = \ep(z+h)\ {\rm or}\ x = \e z\ \ {\rm for\ some}\ z \in \mathbb{Z}\bigr\}.
\end{equation}
In this section we study the behaviour with respect to $\e$ of the eigenvalues $\lambda_\e$ and  eigenfunctions $u_\e,$ using asymptotic expansions. We show that, up to the leading order, the values of $\lambda_\e$ are described by an eigenvalue of the weighted Neumann-Laplacian on the defect $D$, see (\ref{limnue}) below. More precisely, we show that for each eigenvalue $\l_0$ of (\ref{limnue}) in a gap of $\bigcup_\theta \sigma(A_\theta)$, there exists a sequence of eigenvalues $\lambda_\e$ of (\ref{formale1}) converging to $\lambda_0$. However, it remains unclear whether every accumulation point of $\l_\e$ inside a gap of $\bigcup_\theta \sigma(A_\theta)$ belongs to the spectrum of (\ref{limnue}). We address this question in Section \ref{sec:uppersemicont}, where we argue that the eigenmodes $u_\ep$ are asymptotically localised to the defect. The latter observation implies compactness of the sequences of eigenmodes $u_\ep$, thus establishing asymptotic one to one correspondence between the eigenvalues of (\ref{formale1}) and (\ref{limnue}) in the gaps of $\bigcup_\theta \sigma(A_\theta)$.

We seek asymptotic expansions for the eigenvalues $\lambda_\varepsilon$ and eigenfunctions $u_\varepsilon$ of \eqref{formale1}--\eqref{formale3} in the form
\begin{equation}
\label{lambdaexpan}
\lambda_\e = \lambda_0 + \e\l_1+ \ep^2 \lambda_2 \ldots,
\end{equation}
with 
\begin{equation}
\label{ndexp}
u_\e(x)=\left\{\begin{array}{ll}u_0(x) + \e u_1(x) + \e^2 u_2(x) + \ldots, \qquad x \in (d_-,d_+), \\[0.5em] 
w_0(\tfrac{x}{\e})  + \e^2 w_2(\tfrac{x}{\e}) + \ldots,  \qquad x \in (-\infty,d_-)\cup(d_+,\infty).\end{array}\right.
\end{equation}
We assume that functions $w_{2i}$, $u_i$, $i=0,1,2,\ldots$, are continuous.
Substituting \eqref{lambdaexpan}, \eqref{ndexp} into \eqref{formale1} and \eqref{formale2} and equating the $\ep^0$-coefficient  on the defect gives
\begin{equation}\label{limnue}
\left\{ \begin{aligned}
-(a_D u_0')' &= \lambda_0 \rho_D u_0 \qquad & {\rm on}\ (d_-,d_+), \\[0.3em]
a_D u_0' |_D &= 0 \qquad & {\rm on}\ \{d_-, d_+\},
\end{aligned} \right.
\end{equation}
that is, $\lambda_0$ is an eigenvalue of the weighted Neumann-Laplace operator $A_{{\rm N}, D}$ on the defect, {\it cf.} (\ref{A_N,D}). Note that this is true regardless of whether $d_-,$ $d_+$ belong to $\Omega^\e_1$ or $\Omega^\e_0$. We fix $u_0$ by setting $\|u_0\|_{L^2_{\rho_D}(D)}=1$.

For $c\in{\mathbb R}$, let  $\lfloor c \rfloor_\e$ and $\lceil c \rceil_\e$ denote the largest integer $z$ such that $\e z  \le c$ and the smallest integer $z$ such that $c\le \e z,$ respectively.  Substituting \eqref{lambdaexpan}, \eqref{ndexp} into \eqref{formale1},  \eqref{formale3}  and comparing the coefficients for different powers of $\e$ in the resulting expression yields
\begin{equation}
\label{uoe0} \left\{
\begin{aligned}
-(a_1w_0')'   = 0, & \hspace{.5cm} & {\rm on}\  Y_1+z , \\[0.3em]
 a_1w_0' \big \vert_{Y_1+z}  = 0,  & \hspace{.5cm} & {\rm on}\ \{z+h, z+1\},  
\end{aligned} \right.
\end{equation}
and 
\begin{equation}
\label{uoe1} \left\{
\begin{aligned}
-(a_0w_0')'   = \lambda_0 \rho_0 w_0, & \hspace{.5cm} & {\rm on}\  Y_0+z  ,  \\[0.35em]
-(a_1w_2')'   = \lambda_0 \rho_1 w_0, & \hspace{.5cm} & {\rm on}\  Y_1+z , \\[0.3em]
a_1w_2' \big \vert_{Y_1+z}    = a_0 w_0'\big \vert_{Y_0+z} ,  & \hspace{.5cm} & {\rm on}\ \{z+h, z+1\},
\end{aligned} \right.
\end{equation}
for all  
\begin{equation}
z \in \mathcal{I}_\e : = \bigl\{ z \in \mathbb{Z} :  z \ge  \lceil d_+ \rceil_\e \text{ or }  z \le   \lfloor d_- \rfloor_\e - 1 \bigr\}.
\label{triangle}
\end{equation}
The assertion \eqref{uoe0} implies that $a_1w_0' \equiv 0$ on $ Y_1+z  $ and therefore $w_0$ is constant on each such interval. By the second equation of \eqref{uoe1}, and the fact $w_0$ is constant on each interval $ Y_1+z $, the function $a_1w_2'$ has the form
\begin{equation}
(a_1w_2')(y) = (a_1w_2')(z+h)-\lambda_0 w_0(z+h)\int\limits_{z+h}^y \rho_1,\ \ \ \ \ y\in Y_1+z .
\label{star}
\end{equation}
Combining (\ref{star}), the fact $w_0$ is constant on $Y_1+z$ and the first and last equations of \eqref{uoe1} implies that for all $z\in\mathcal{I}_\e$ one has
\begin{equation}
\label{effectivetransferproblem}
\left\{\begin{array}{lcc}
-(a_0w_0')'   = \lambda_0 \rho_0 w_0, \qquad {\rm on}\  Y_0+z  , & \hspace{1cm}&\ \\[0.9em]
w_0 \equiv w_0(z+h) = w_0(z+1) ,  \qquad {\rm on}\  Y_1+z , & & \ \\[0.6em]
(a_0w_0')\big( (z+1)^+\big) - (a_0w_0') \big( (z+h)^-\big) = - \lambda_0 w_0(z+h) \int\limits_{Y_1} \rho_1. & & \ 
\end{array}  \right.
\end{equation}
The problem (\ref{effectivetransferproblem}) fully governs the behaviour of $w_0$ in $\RR \backslash ( \lfloor d_{-} \rfloor_\ep -1 ,  \lceil d_+ \rceil_\ep )$. We can utilise the fundamental system $(v_1,v_2)$  from Section \ref{sec:transfer} to quantitatively characterise $w_0$. Indeed, since in each cell $Y+z$ any solution to the first equation in \eqref{effectivetransferproblem} is a linear combination of $v_1$ and $v_2,$ one has
\begin{equation}
\label{u0+}
w_0(y) = \left\{ \begin{array}{lr}
 l_z v_1(y-z) + m_z v_2(y-z), & y \in  Y_0+z  , \\[0.5em]
  l_z v_1(h) + m_z v_2(h), & y \in  Y_1+z  ,
\end{array} \right.
\end{equation}
for constants $l_z, m_z$, $z \in \mathcal{I}_\e,$ where the expression on $Y_1+z$ follows from the second condition in (\ref{star}). Using the continuity of $w_0$ and the  jump of the co-derivative condition from (\ref{effectivetransferproblem}) it is not difficult to derive the following recurrence relation:
\begin{equation}\label{u0=2}
\left( \begin{matrix}
l_{z+1} \\ m_{z+1}
\end{matrix}  \right) =  \left( \begin{matrix}
v_1(h) & & v_2(h) \\[5pt]
(a_0 v_1')(h) - \lambda_0 v_1(h) \int_{Y_1} \rho_1 & & (a_0 v_2')(h) - \lambda_0 v_2(h) \int_{Y_1} \rho_1
\end{matrix}  \right) \left( \begin{matrix}
l_{z} \\ m_{z}
\end{matrix}  \right).
\end{equation}
Now, recalling the Wronskian property \eqref{7.81}, we find that the characteristic polynomial $q$ of the matrix in (\ref{u0=2}) is ({\it cf.} (\ref{qdef}))
$$
q(\mu)=\mu^2 - \Big( v_1(h) + (a_0v_2')(h) - \lambda_0 v_2(h) \int_{Y_1} \rho_1 \Big) \mu + 1.
$$
The roots $\mu_1$, $\mu_2$ of $q$ satisfy the identity $\mu_1 \mu_2 = 1$ and the nature of $w_0$ as it varies from one period to the next is determined by the quantity  $v_1(h) + (a_0v_2')(h) - \lambda_0 v_2(h) \int_{Y_1} \rho_1  $. Namely, if ({\it cf.} (\ref{specineq})) 
$$
\left| v_1(h) + (a_0v_2')(h) - \lambda_0 v_2(h) \int_{Y_1} \rho_1 \right|  \le 2,
$$ then the roots $\mu_1$, $\mu_2$ are complex conjugate with $| \mu_1 | = | \mu_2 | =1$ and solutions $w_0$ are described by the linear span of two quasi-periodic functions with phase difference $\pi$. In Section \ref{sec:limspectrum} we demonstrated that $\lambda_0$ satisfies this constraint if and only if $\lambda_0$ belongs to the limit spectrum 
$$
\lim_\e \sigma(A^\e) = \bigcup_\theta \sigma(A_\theta).
$$
For $\lambda_0$ in the gaps of this limit spectrum, {\it i.e.} when $\lambda_0$ satisfies the inequality
$$
\bigg|  v_1(h) + (a_0v_2')(h) - \lambda_0 v_2(h) \int_{Y_1} \rho_1 \bigg| > 2,
$$
the roots $\mu_1, \mu_2$ satisfy $|\mu_1| <1$ and $|\mu_2|>1$. For such $\lambda_0$,  we can construct ``unstable" solutions, one of which decays and the other grows. Indeed, denoting by $\varkappa_1$ and $\varkappa_2$ the eigenvectors corresponding to $\mu_1$ and $\mu_2$ respectively, we find in  the interval $[\lceil d_+ \rceil_\e, \infty)$ that $w_0$ given by \eqref{u0+}, \eqref{u0=2} satisfies $w_0(y+1) = \mu_j w_0(y)$ if $(l_{\lceil d_+ \rceil_\ep} , m_{\lceil d_+ \rceil_\ep} ) = \varkappa_j,$ $j=1,2.$ Similarly, in the interval $(-\infty, \lfloor d_- \rfloor_\e]$,  one has $w_0(y) = \mu_j^{-1} w_0(y-1)$ if $(l_{\lfloor d_- \rfloor_\e-1} , m_{\lfloor d_- \rfloor_\e-1} ) =\varkappa_j,$ $j=1,2.$ For $w_0$ to decay to the left and right of the defect, we set  $(l_{\lceil d_+ \rceil_\ep} , m_{\lceil d_+ \rceil_\ep} ) = \varkappa_1$ and  $(l_{\lfloor d_- \rfloor_\e-1} , m_{\lfloor d_- \rfloor_\e-1} ) = \varkappa_2$. In this way we ensure that 
\begin{equation}\label{4.26}
\begin{aligned}
&	w_0(y+1) = \mu_1 w_0(y) & \mbox{for } & y \in [\lceil d_+ \rceil_\e , \infty), \\
& w_0(y-1) = \mu_2^{-1} w_0(y) = \mu_1 w_0(y) & \mbox{for } & y \in (-\infty, \lfloor d_- \rfloor_\e].
\end{aligned}
\end{equation} 
We extend $w_0$ into the cells as follows:
\begin{equation*}
\begin{aligned}
&	w_0(y) = \mu_1^{-1} w_0(y+1) & \mbox{for } & y \in I_{\rm r} : = ( \tfrac{d_+}{\ep}, \lceil d_+ \rceil_\e\big),\\
& w_0(y) = \mu_1^{-1} w_0(y-1) & \mbox{for } & y \in I_{\rm l} := (\lfloor d_- \rfloor_\e, \tfrac{d_-}{\ep}).
\end{aligned}
\end{equation*} 
We  choose  $\varkappa_1$ and $\varkappa_2$ so that the constructed $w_0$ matches the value of $u_0$ at the ends of $D$: 
\[
w_0( d_+/\e) = u_0(d_+), \quad w_0( d_- /\e) = u_0(d_-).
\]
Note that the normalisation factor for $\varkappa_1,$ $\varkappa_2$ depends on 
 $\e$ in general, but it is nevertheless bounded uniformly in $\e$.

The second equation and third equations of (\ref{uoe1}) determine $w_2$ in the stiff component up to an arbitrary additive constant in each interval $Y_1+z$, $z\in \mathcal I_\e,$ and in the stiff intervals  $( I_{\rm l} \cup I_{\rm r} ) \cap \ep^{-1}\Omega^\ep_1$
 adjacent to  $D$. In the intervals $Y_1+z$, $z\in \mathcal I_\e,$ we choose this constant so that 
\begin{equation}\label{4.27}
\begin{aligned}
&w_2(h+z)=0 &\mbox{ if } \e(h+z) \geq d_+, 
\\
&w_2(1+z)=0 &\mbox{ if } \e(h+z)\leq d_-. 
\end{aligned}
\end{equation}
In the intervals $( I_{\rm l} \cup I_{\rm r} ) \cap \ep^{-1}\Omega^\ep_1$ we choose the value of the constant so that 
\begin{equation}\label{4.28}
w_2(d_-/\e)=w_2(d_+/\e)=0.
\end{equation}
In the soft component $Y_0+z$, $z\in \mathcal I_\e$ we do not require $w_2$ to satisfy any equation. Instead we make a specific choice of $w_2$ as follows. Let $f\in C^\infty_0(Y_0)$ be a positive function, then we define 
\[
w_2(z+y) := w_2(z) + c_z \int_0^y \frac{f}{a_0},\qquad y\in Y_0, \ z \in \mathcal I_\e,
\]
where the coefficients $c_z$ are chosen so that $w_2$ is continuous on $\R\setminus D$. Thus, we have 
\[
(a_0 w_2')(z^+) = (a_0 w_2')((z+h)^-) = 0,  \quad z\in \mathcal I_\e.
\]
Finally,  conditions  (\ref{4.27}), (\ref{4.28}) imply that we can extend $w_2$ by zero into the soft intervals in the cells adjacent to $D:$
\begin{equation*}
w_2 \equiv 0 \,\, \mbox{ in } \left[ \bigl(\lfloor d_- \rfloor_\e , \lfloor d_- \rfloor_\e+h\bigr)\cup\bigl( \lfloor d_+ \rfloor_\e ,  \lfloor d_+ \rfloor_\e+ h\big) \right]\setminus \e^{-1}D.
\end{equation*}

It remains to define $u_1$ on $D$ so that it vanishes on the boundary of $D$ and so that its co-derivative matches the co-derivative of $w_0(x/\e) + \e^2 w_2(x/\e).$ 
We require 
\begin{equation*}
	\begin{aligned}
(a_D u_1')\big((d_+)^-\big) = J_1 : = \left\{
	\begin{aligned}
	& (a_0 w_0')\big((d_+/\e)^+\big), & \mbox{ if }d_+/\e \in\big[\lfloor d_+ \rfloor_\e , \lfloor d_+ \rfloor_\e+h\big),
	\\
	& (a_1 w_2')\big((d_+/\e)^+\big), & \mbox{ if } d_+/\e \in \big[\lfloor d_+ \rfloor_\e +h, \lceil d_+ \rceil_\e\big),
	\end{aligned}
	\right. 
	\\ 
	(a_D u_1')\big((d_-)^+\big) = J_2 : = \left\{
	\begin{aligned}
	& (a_0 w_0')\big((d_-/\e)^-\big), & \mbox{ if }d_-/\e \in\big[\lfloor d_- \rfloor_\e , \lfloor d_- \rfloor_\e+h\big),
	\\
	& (a_1 w_2')\big((d_-/\e)^-\big), & \mbox{ if } d_-/\e \in \big[\lfloor d_- \rfloor_\e +h, \lceil d_- \rceil_\e\big).
		\end{aligned}
			\right.
		\end{aligned}
\end{equation*}
In order to fulfil the above conditions we take a smooth cut-off function $\chi$ such that $\chi(x)=0, \, x\leq d_-+\delta$, $\chi(x)=1, \, x\geq d_+-\delta$, for a sufficiently small $\delta>0$, and define
\begin{equation*}
u_1(x):= J_1\chi(x) \int_{d_+}^x a_D^{-1} + J_2(1-\chi(x))\int_{d_-}^x a_D^{-1}, \qquad x \in \RR.
\end{equation*}

Suppose now that $\lambda_0\in\sigma\bigl(A_{{\rm N}, D}\bigr)\backslash\bigl(\bigcup_\theta\sigma(A_\theta)\bigr).$ The construction described above guarantees that the function
\begin{equation}\label{4.33}
u_{\e, \rm ap} (x): = \left\{ 
\begin{aligned}
& u_0(x) +\e u_1(x), & x\in D,
\\
& w_0(x/\e) + \e^2 w_2(x/\e), & x\in \R\setminus D,
\end{aligned}
\right.
\end{equation}
is continuous and has a continuous co-derivative $a^\e_D u_{\e, \rm ap}'$, implying that $u_{\e, \rm ap}$ belongs to the domain of the operator $A_D^\e$. Moreover, it is not difficult to see that 
\begin{equation}\label{4.29}
\begin{aligned}
\|w_0(\cdot/\e)\|_{L^2_{\rho^\e_D}(\R\setminus D)} &\leq  \e^{1/2} || w_0 ||_{L^2_\rho(\RR \backslash \ep^{-1} D)},
\\
\|(a_0 w_2')'(\cdot/\e)\|_{L^2_{\rho^\e_D}(\R\setminus D)} + \|w_2(\cdot/\e)\|_{L^2_{\rho^\e_D}(\R\setminus D)} &\leq C \|w_0(\cdot/\e)\|_{L^2_{\rho^\e_D}(\R\setminus D)},
\\
\|(a_D u_1')'\|_{L^2_{\rho^\e_D}(D)} + \|u_1\|_{L^2_{\rho^\e_D}(D)}& \leq C,
\end{aligned}
\end{equation}
for some constant $C>0$.

It follows from the spectral theorem for self-adjoint operators (see {\it e.g.} \cite{Birman_Solomjak}) that for all functions $f \in {\rm dom}\bigl(A_D^\varepsilon\bigr)\subset L^2_{\rho^\e_D}(\R)$ such that $\Vert f\Vert_{L^2_{\rho^\e_D}(\mathbb R)}=1,$ one has 
\begin{equation*}
{\rm dist}{\big( \lambda_0 , \sigma\left(A^\e_D\right)\big)} \le {\bigl\Vert (A_D^\e - \lambda_0)f\bigr\Vert_{L^2_{\rho^\e_D}(\mathbb R)}}.
\end{equation*}
Straightforward calculations show that 
\begin{equation*}
(A_D^\e - \lambda_0) u_{\e, \rm ap} = \left\{
\begin{aligned}
& - \e (a_D u_1')'(x) - \e \l_0\rho_D u_1(x), & x &\in D,
\\
& -\e^2 (a_0 w_2')'(x/\e) - \e^2 \l_0 \rho_0 w_2(x/\e), & x &\in \Omega_0^\e\setminus D,
\\
& - \e^2 \l_0 \rho_1 w_2(x/\e), & x &\in \Omega_1^\e\setminus D.
\end{aligned}
\right.
\end{equation*} 
Then (\ref{4.29}) readily implies there exists $C>0$ such that
\begin{equation}\label{4.30}
\bigl\|(A_D^\e - \lambda_0) u_{\e, \rm ap}\bigr\|_{L^2_{\rho^\e_D}(\R)} \leq C \e.
\end{equation}

We establish the following result, which implies Claim 1 of Theorem \ref{mainthm}. In particular, the second estimate in (\ref{2.6}) follows from (\ref{4.33}), (\ref{4.29}) and (\ref{eig_conv_est}) below.
\begin{theorem}
\label{lem:locmod}
Suppose  that $\lambda_0\in\sigma\bigl(A_{{\rm N}, D}\bigr)\backslash\bigl(\bigcup_\theta\sigma(A_\theta)\bigr).$ 
\begin{enumerate}[1.]
	\item{ There exists $C_1>0$, independent of $\e$,  such that 
		$$
		{\rm dist}\big( \lambda_0, \sigma(A^\e_D) \big) \le C_1 \e.
		$$}
	\item{ For sufficiently small $\ep$ there exist (simple) eigenvalues $\lambda_\e$ of $A^\e_D$ 
	such that $|\l_\e - \l_0|\leq C_1 \e$. 
	}
	\item{	For sufficiently small $\ep$ the function $u_{\e, \rm ap}$ is an approximate eigenfunction of $A^\e_D:$ there exists a constant $C_2>0$ independent of $\e$ such that	
		\begin{equation}
		\bigl\Vert u_{\e, \rm ap} - u^\e\bigr\Vert_{L^2_{\rho_D^\varepsilon}(\RR)} \le C_2 \e,
		\label{eig_conv_est}
		\end{equation}
	where $u^\e$ is the  
	eigenfunction of $A^\e_D$ corresponding to the eigenvalue $\lambda^\e.$ 

		}
\end{enumerate}
\end{theorem}
\begin{proof}
Claim 1 of the theorem follows from (\ref{4.30}) and the fact that $\|u_{\e, \rm ap}\|_{L^2_{\rho^\e_D}(\R)}\to\|u_{0}\|_{L^2_{\rho_D}(D)} =1$ as $\e\to 0,$ due to (\ref{4.29}).  Claim 2 follows by noting that the essential spectra of $A^\e_D$ and $A^\e$ coincide, and that $\sigma(A^\e)= \sigma_{\rm ess}(A^\e)$ converges to ${ \bigcup_\theta \sigma(A_\theta) }$, as $\e \rightarrow 0$, to which $\lambda_0$ does not belong. To prove claim 3, one can argue as in \cite{VL}, or \cite[Section 11.1]{JKO}. Namely, it follows from  (\ref{4.30}) and a spectral decomposition of $u_{\e, \rm ap}$ with respect to the operator $A^\e_D$ that there exists an $\varepsilon$-independent constant $C_2>0$ 
and $c_j^\e\in{\mathbb R}$ such that
$$
\Bigl\Vert u_{\e, \rm ap} - \sum_{j \in J_\e} c_j^\e u_{\e,j}\Bigr\Vert_{L^2_{\rho_D^\varepsilon}(\RR)} \le C_2 \e,
$$
where for each $\ep$,  $J_\e : = \left\{ j : \left\vert \lambda_{\e,j} - \lambda_0 \right\vert \le C_2 \e \right\}$ is a finite set of indices and $u_{\e,j}$ are $L_{\rho^\e_D}^2(\R)$-normalised eigenfunctions of $A^\e_D$ with eigenvalue $\lambda_{\e,j}$. Next, the compactness property demonstrated in Theorem \ref{th5.1}, see next section, implies that there is exactly one sequence of simple eigenvalues $\lambda_\varepsilon$ converging to $\lambda_0,$ hence (\ref{eig_conv_est}) holds.
\end{proof}

%

\section{Spectral completeness of defect eigenvalues and localisation of eigenmodes}
\label{sec:uppersemicont}
The method of asymptotic expansions allows us to show that for any eigenvalue $\l_0$ of $A_{{\rm N}, D}$, cf. (\ref{A_N,D}),  in a gap of $\bigcup_\theta \sigma(A_\theta)$ there exists a sequence of eigenvalues of $A^\e_D$ converging to $\lambda_0$. The converse statement requires a compactness argument for a corresponding sequence of eigenfunctions of $A^\e_D$. In this section we use functional analytic techniques, which, unlike Section \ref{sec:decay}, do not rely on the one-dimensional nature of the problem, to show a decay of the eigenfunctions of $A^\e_D$ outside the defect sufficient to imply the compactness of sequences of eigenfunctions with eigenvalues accumulating in the gaps of $\bigcup_\theta \sigma(A_\theta)$. 

\begin{theorem}\label{th5.1}
	Let $\lambda_\ep$ be an eigenvalue sequence of $A^\e_D$, $u_\e$ be a corresponding sequence of $L^2(\mathbb{R})$-normalised eigenfunctions, and suppose that $\lambda_\e \rightarrow \lambda_0 \in \mathbb{R} \backslash {\bigcup_\theta \sigma(A_\theta)}$ as $\e \rightarrow 0$. Then $\l_0$ is an eigenvalue of $A_{{\rm N}, D}$ and up to a subsequence
	\begin{equation*}
	\begin{aligned}
	u_\e \to u_0\mbox{ strongly in } L^2(\R), & \qquad & u_\e \rightharpoonup u_0 \mbox{ weakly in } H^1(D),
	\end{aligned}
	\end{equation*}
	where $u_0$ is an eigenfunction corresponding to $\l_0$, extended by zero outside the defect $D$.
\end{theorem}
\begin{proof}
The main ingredient of the proof is demonstrating that  the eigenfunction sequences $u_\e$ localise to the defect in the  sense that
\begin{equation}
\label{upperclaim1}\lim_{\e \rightarrow 0}|| \chi_{\e,\alpha} u_\e ||_{L^2(\mathbb{R})} = 0 \quad \forall \, \alpha \in (0,1),
\end{equation}
where  $\chi_{\ep,\alpha}: C^\infty(\mathbb{R}) \rightarrow [0,1],$ $\varepsilon>0$, is any smooth cut-off function such that 
$$
\chi_{\e,\alpha} = \left\{ \begin{array}{ll}
0 & {\rm in}\ \ D, \\[0.4em]
1 & {\rm in}\ \ (-\infty, d_--\ep^{\alpha}]\cup [d_+ + \ep^\alpha,\infty).
\end{array} \right.
$$
Additionally, $\chi_{\e,\alpha}$ is constant on each connected component of $\Omega^1_\e$ and satisfies the bound $\sup\limits_\e \e^\alpha \Vert \chi_{\e,\alpha}' \Vert_{L^\infty(\mathbb R)} < \infty$.
The assertion \eqref{upperclaim1} is an immediate consequence of the following lemma, that we demonstrate below.
\begin{lemma}
	\label{lem:eigmoderep} Consider a sequence $\lambda_\e \in [0,\infty)$, $u_\e \in L^2(\mathbb{R})$, $|| u_\e ||_{L^2(\mathbb{R})} =1$, such that  $A^\e_D u_\e = \lambda_\e \rho^\e_Du_\e$. If the convergence $\lambda_\ep\to\lambda_0 \in \mathbb{R} \backslash {\bigcup_\theta \sigma(A_\theta)}$ holds as $\varepsilon\to0,$ then there exist sequences $v_\e$, $w_\e \in H^1(\mathbb{R})$ such that $u_\e = v_\e + w_\e$ 
with the following properties:

1) One has 
$v'_\e \equiv  0$ on $\Omega^\e_1 \backslash D;$

2) The sequence $v_\e$ is localised to defect in the sense of \eqref{upperclaim1};

3) There exists a constant $C>0$ such that
\begin{equation}
	\label{2.24}
	\|w_\e\|_{L^2(\R)}\leq C \e^2,\qquad\|w_\e'\|_{L^2(\R)}\leq C \e.
\end{equation}

\end{lemma}
Let us prove that $\lambda_0 \in \sigma(A_{{\rm N},D})$ under the assumption that Lemma \ref{lem:eigmoderep} holds. By substituting $\varphi = u_\e$ in the eigenvalue problem for the operator $A^\e_D$ ({\it cf.}  \eqref{weakdefectform})
\begin{equation*}
\int_D a^\e_D u_\e' \overline{\varphi'}  =  \l_\e \int_\R \rho^\e_D u_\e\overline{\varphi} \qquad \forall \varphi \in H^1(\R),
\end{equation*}
 and utilising the boundedness of $\lambda_\e$, the uniform positivity and boundedness of $a_j,$ $\rho_j,$ $j=1,2,$ $a_D$ and $\rho_D,$ we establish the estimates
\begin{equation}\label{1.5}
\begin{aligned}
\sup_\e || u_\e ||_{H^1(D)} < \infty,& \qquad & \sup_\e\|u_\e'\|_{L^2(\Omega_1^\e\backslash D)} < \infty, & \qquad & \sup_\e\|\e u_\e'\|_{L^2(\Omega_0^\e\backslash D)} < \infty.
\end{aligned}
\end{equation}
By \eqref{1.5}, it is clear that a subsequence of $u_\varepsilon$ converges weakly in $H^1(D)$.
Now, by Lemma \ref{lem:eigmoderep} and the identity $u_\e = \chi_{\e,\alpha} v_\e + (1- \chi_{\e,\alpha}) v_\e + w_\e$ we find that $u_\e$ strongly converges to zero in $L^2(\R \backslash D)$. Therefore, there exists  $u_0 \in L^2(\mathbb{R})$, $u_0 \equiv 0$ in $\mathbb{R} \backslash D$, such that up to a subsequence
\begin{equation*}
\begin{aligned}
u_\e \to u_0\mbox{ strongly in } L^2(\R), & \qquad & u_\e \rightharpoonup u_0 \mbox{ weakly in } H^1(D).
\end{aligned}
\end{equation*}
Moreover, Lemma \ref{lem:eigmoderep} implies that
\begin{equation*}
a_1(\tfrac{\cdot}{\ep})u_\e' \to 0 \mbox{ strongly in } L^2(\Omega^\e_1 \backslash D).
\end{equation*}
Therefore, for fixed $\varphi \in H^1(\mathbb{R})$, we can pass to the limit in \eqref{weakdefectform}, recalling the identity
\begin{equation}
\int_{\mathbb R} a^\e_D u_\e' \overline{\varphi'} = \int_D a_D u_\e' \overline{\varphi'} + \int_{\Omega^\e_1 \backslash D} a_1 (\tfrac{x}{\e}) u_\e' \overline{\varphi'}  + \int_{\Omega^\e_0 \backslash D} \ep^2a_0(\tfrac{x}{\e}) u_\e' \overline{\varphi'}
\end{equation}
to find that
\begin{equation}\label{4.38}
\int_{D} a_D u_0'\overline{\varphi'}= \lambda_0  \int_{D} \rho_D  u_0 \overline{\varphi}.
\end{equation}
Finally, by the arbitrariness of $\varphi$ deduce that $\lambda_0 \in \sigma(A_{{\rm N},D})$ and $u_0$ is the corresponding eigenfunction.

\end{proof}

\begin{corollary}
Claim 2 of Theorem \ref{mainthm} holds.
\end{corollary}

We now prove Lemma \ref{lem:eigmoderep}.
\begin{proof}[Proof of Lemma \ref{lem:eigmoderep}]

 We start by constructing the representation $u_\e$ as the sum of  $v_\e$ and $w_\e$ 
 as follows. On the defect $D=(d_-,d_+)$, we set $v_\e = u_\e$. On each connected component of $\Omega_1^\e$, except for the intervals adjacent to the defect, we define $v_\e$ as
\begin{equation}
v_\e(x) := 
\frac{1}{\e (1-h)} \int_{\e (Y_1 + z)}u_\e, \qquad x \in \e(Y_1+z),\ \ z \in \mathcal{I}_\e,
\end{equation}
where ${\mathcal I}_\varepsilon$ is defined by (\ref{triangle}).
If necessary, we extend $v_\e$ continuously by constant from $D$ into the stiff region adjacent to the defect, {\it i.e.} $\Omega^\e_1 \cap \bigl(  \big(\min\bigl\{ d_-, \ep(\lfloor d_- \rfloor_\e+h)\bigr\}, d_-\bigr] \cup \bigl[\max\bigl\{\ep(\lfloor d_+ \rfloor_\e+ h), d_+ \bigr\}, \lceil d_+ \rceil_\e \big) \big)$. Thus $v_\e$ is defined everywhere except the soft component $\Omega^\e_0 \backslash D$, and is piecewise constant on the stiff component $\Omega^\e_1 \backslash D$. To define $v_\e$ on  $\Omega^\e_0 \backslash D$ we ensure that the difference $w_\e:= u_\e - v_\e, \,\, x \in  \Omega_1^\e \backslash D,$
is extended into the soft component $\Omega^\e_0 \backslash D$ so that $w_\e \in H^1(\R)$ and  satisfies
\begin{equation}\label{1.6}
\bigl(a_0(\tfrac{\cdot}{\e}) w_\e'\bigr)' = 0, \qquad \mbox{ on } \Omega_0^\e \backslash D.
\end{equation}
Thus we have $u_\e = v_\e + w_\e$, where $v_\e, w_\e \in H^1(\R)$ with $v' \equiv 0$ on $\Omega^\e_1 \backslash D$ and $w_\e \equiv 0$ in $D$. 

We first prove \eqref{2.24}. By construction, for each $z \in \mathcal{I}_\e$, the function $w_\e$ has zero mean value on the interval $ \e(Y_1+z)$ and it is clear, for example by an application of the fundamental theorem of calculus, that for each $z \in \mathcal{I}_\e$ one has
\begin{equation}\label{1.11}
\bigl|w_\e(x)\bigr|^2 \leq  \e(1-h)\int_{\e (Y_1 + z)} |w_\e'|^2, \qquad x \in \e(Y_1+z).
\end{equation}
A version of the same argument implies that since $w_\e \equiv 0$ on $(d_-,d_+)$, on $I_\e : = (\e \lfloor d_- \rfloor_\e, \e \lceil d_+ \rceil_\e) $ we have 
\begin{equation}
\label{1.11extra}
\bigl|w_\e(x)\bigr|^2 \leq  \e \int_{I_\e \backslash D} |w_\e'|^2, \qquad x \in I_\e.
\end{equation}
Moreover, if the soft component touches the defect on the right, {\it i.e.} if $d_+ < \e (\lfloor d_+ \rfloor_\e + h)$ then
\begin{equation}\label{5.42}
  \begin{aligned}
 |w_\e(x)| \leq(\e \lfloor d_+ \rfloor_\e + \e h - d_+)^{-1}\int_{( d_+,  \e \lfloor d_+ \rfloor_\e + \e h)} |w_\e'|^2,  \quad x\in ( d_+,  \e \lfloor d_+ \rfloor_\e + \e h)
 \end{aligned}
\end{equation}
 and  if the soft component touches the defect on the left, {\it i.e.} $d_- \leq \e (\lfloor d_- \rfloor_\e + h)$, then
 
 \begin{equation}\label{5.43}
 |w_\e(x)| \leq(d_- - \e \lfloor d_- \rfloor_\e )^{-1}\int_{(\e \lfloor d_- \rfloor_\e,d_-)} |w_\e'|^2,  \quad x\in ( \e \lfloor d_- \rfloor_\e,d_-).
 \end{equation}
 
On the soft component $\Omega^\e_0 \backslash I_\e = \bigcup_{z \in \mathcal{I}_\e} \e(Y_0+z)$, we note that since $w_\e$ solves \eqref{1.6}, the 
maximum 
principle implies
\begin{equation*}
\sup_{(\e z, \e(z+ h))}|w_\e| = \max\bigl\{\bigl|w_\e(\e z)\bigr|,\, \bigl|w_\e(\e(z + h)) \bigr|\bigr\} \qquad \forall z \in \mathcal{I}_\e.
\end{equation*}
This fact, along with inequalities \eqref{1.11} and  \eqref{1.11extra}, 
implies that
\begin{equation*}
\|w_\e\|_{L^2(\e(Y_0+z))} \leq \e\,  \max\big\{\|w_\e'\|_{L^2(\e(Y_1+z))},\,\|w_\e'\|_{L^2(\e(Y_1+z-1))}\big\}.
\end{equation*}
Putting the above inequalities together, it follows that
\begin{equation}\label{1.9}
\|w_\e\|^2_{L^2(\R)} =  \int_{I_\e } | w_\e |^2 + \sum_{z \in \mathcal{I}_\e } \int_{\e (Y_0 + z)}| w_\e |^2 + \sum_{z \in \mathcal{I}_\e } \int_{\e (Y_1 + z)} | w_\e |^2  \leq  2 \e^2\|w_\e'\|^2_{L^2(\Omega_1^\e\backslash D)}.
\end{equation}

Straightforward calculations show that due to (\ref{1.6}) we have on the soft component
\begin{equation*}
\sup_{\e(Y_0+z)} |w_\e'| \leq (\e h)^{-1}||a_0||_{L^\infty(Y_0)} || a_0^{-1} ||_{L^\infty(Y_0)}\Bigl(\bigl|w_\e(\e z)\bigr|+ \bigl|w_\e(\e z + \e h)\bigr|\Bigr). \\
\end{equation*}
Similarly, if the soft component touches the defect on the right, {\it i.e.} if $d_+ < \e (\lfloor d_+ \rfloor_\e + h)$ then
$$
\begin{aligned}
	\sup_{( d_+,  \e \lfloor d_+ \rfloor_\e + \e h)} |w_\e'| \leq(\e \lfloor d_+ \rfloor_\e + \e h - d_+)^{-1}||a_0||_{L^\infty(Y_0)} || a_0^{-1} ||_{L^\infty(Y_0)} \,  \bigl|w_\e (\e \lfloor d_+ \rfloor_\e + \e h)\bigr|, 
\end{aligned}
$$
and  if the soft component touches the defect on the left, {\it i.e.} $d_- \leq \e (\lfloor d_- \rfloor_\e + h)$, then
$$
\begin{aligned}
\sup_{( \e \lfloor d_- \rfloor_\e,d_-)} |w_\e'| \leq(d_- - \e \lfloor d_- \rfloor_\e )^{-1}||a_0||_{L^\infty(Y_0)} || a_0^{-1} ||_{L^\infty(Y_0)} \,  \bigl|w_\e(\e \lfloor d_- \rfloor_\e)\bigr|.
\end{aligned}
$$
Consequently, from (\ref{1.11})--(\ref{5.43}) and the above assertions, we obtain 
\begin{equation}\label{1.19}
\|w_\e'\|_{L^2(\Omega_0^\e \backslash D)} \leq C \|w_\e'\|_{L^2(\Omega_1^\e \backslash D)}.
\end{equation}

It remains to  bound $w_\e'$ on the stiff component $\Omega^\e_1 \backslash D$, which in combination with (\ref{1.9}) and  (\ref{1.19})  yields the  estimates (\ref{2.24}). To this end, note that by setting $\varphi = w_\e$ in \eqref{weakdefectform}, using the identity $u_\e = v_\e + w_\e$ and the facts that $v_\ep' = 0 $ in $\Omega^\e_1 \backslash D$ and  $w_\e \equiv 0$ in $D$, we have
\begin{equation*}
\int_{\Omega_0^\e \backslash D} \e^2 a_0(\tfrac{\cdot}{\e}) u_\e' w_\e' +\int_{\Omega_1^\e \backslash D} a_1(\tfrac{\cdot}{\e})| w_\e'|^2 dx= \lambda_\e  \int_{\Omega_0^\e \backslash D} \rho_0(\tfrac{\cdot}{\e})   u_\e w_\e + \lambda_\e  \int_{\Omega_1^\e \backslash D} \rho_1(\tfrac{\cdot}{\e})   u_\e w_\e.
\end{equation*}
Hence, by the H\"older inequality we deduce that
$$
|| w_\e' ||_{L^2(\Omega_1^\e \backslash D)}^2 \le \e || \e u_\e' ||_{L^2(\Omega_0^\e \backslash D)}  || w_\e' ||_{L^2(\Omega_0^\e \backslash D)} + C \left( || u_\e ||_{L^2(\R)} || w_\e ||_{L^2(\R)} \right)
$$
for some $C>0,$ and utilising  (\ref{1.5}), \eqref{1.9}, \eqref{1.19}   yields
\begin{equation}\label{w'stiff}
\|w_\e'\|_{L^2(\O_1^\e \backslash D)}\leq C \e.
\end{equation}
Hence, by \eqref{1.9}, \eqref{1.19}, \eqref{w'stiff} and the fact $w \equiv 0$ in $D$, it follows that \eqref{2.24} holds.

We now prove Claim 2. For a fixed $\varphi \in H^1(\R)$ we take a test function $\chi_{\e,\alpha} \varphi$ in \eqref{weakdefectform},  use the identity $u_\e' (\chi_{\e,\alpha}\varphi)' = (u_\e \chi_{\e,\alpha})'\varphi' - u_\e \chi_{\e,\alpha}' \varphi' + u_\e' \chi_{\e,\alpha}'\varphi$ and the decomposition $u_\e = v_\e + w_\e$  to  arrive at the equation
\begin{equation*}
\int_{\mathbb R} \big(a^\e_D (\chi_{\e,\alpha} v_\e)'\varphi' - \lambda_\e \rho^\e_D \chi_{\e,\alpha}v_\e\varphi \big) =   \int_{\mathbb R} \big(\lambda_\e \rho^\e_Dw_\e \chi_{\e,\alpha}\varphi - a^\e_D (w_\e\chi_{\e,\alpha})'\varphi' + a^\e_D \chi_{\e,\alpha}' (u_\e\varphi' - u_\e'\varphi)\big).
\end{equation*}
By inequalities \eqref{2.24}, \eqref{1.5},  the fact that $\chi'_{\e,\alpha} \equiv 0$ on $\Omega^\e_1$, and $\sup\limits_\e \e^\alpha | \chi'_{\e,\alpha} | < \infty$ we can estimate the right-hand side as follows:
\begin{equation*}
\left|\int_{\mathbb R}\big(\lambda_\e \rho^\e_Dw_\e \chi_{\e,\alpha}\varphi - a^\e_D (w_\e\chi_{\e,\alpha})'\varphi' + a^\e_D \chi_{\e,\alpha}' (u_\e\varphi' - u_\e'\varphi)\big)\right|\leq C \e^{1-\alpha} || \varphi ||_{H^1(\R)}.
\end{equation*}
Therefore, one has 
\begin{equation*}
\lim_{\e \rightarrow 0}\sup_{\substack{\varphi \in H^1(\R) \\ || \varphi ||_{H^1(\R)}=1}}\left| \int_{\mathbb R} \big(a^\e_D (\chi_{\e,\alpha}v_\e)'\varphi' - \lambda_\e \rho^\e_D \chi_{\e,\alpha}v_\e\varphi \big) \right\vert= 0,
\end{equation*}
Notice that $(\chi_{\e,\alpha}v_\e)' \equiv 0$ on $\Omega^\e_1$ and $D$,  and therefore
$$
\int_{\mathbb R} \big(a^\e_D (\chi_{\e,\alpha}v_\e)'\varphi' - \lambda_\e \rho^\e_D \chi_{\e,\alpha}v_\e\varphi \big)  = \sum_{z \in \mathbb{Z}}\int_{z}^{z+h} a_0\bigl(\mathcal{R}_\e(\chi_{\e,\alpha}v_\e)\bigr)'\bigl(\mathcal{R}_\e(\varphi)\bigr)'- \lambda_\e \int_{\mathbb R} \rho \mathcal{R}_\e(\chi_{\e,\alpha}v_\e)\mathcal{R}_\e(\varphi),
$$
where $\mathcal R_\e : L^2_{\rho^\e}(\R) \rightarrow L^2_\rho(\R)$ is the unitary transformation $\mathcal R_\e(f)(y) = \e^{1/2} f(\e y)$. It follows that for $z_\e: = \mathcal R_\e(\chi_{\e,\alpha}v_\e)$ one has 
\begin{equation}
\label{3.31}
\lim_{\e \rightarrow 0} \sup_{\substack{\varphi \in H^1(\R) \\ || \varphi ||_{H^1(\R)}=1}} \left| \int_{\Omega_0} a_0 z'_\e\varphi' - \lambda_\e \int_{\mathbb R} \rho z_\e\varphi \right\vert= 0,
\end{equation}
where $\Omega_0 : = \bigcup\limits_{z \in \Z} (Y_0+z)$.

We now argue as in the demonstration of a Weyl's criterion for quadratic forms, see \cite[Appendix]{Krejcirik}, to show the above condition implies that $z_\e$ necessarily converges strongly to zero in $L^2(\R)$. Taking test functions in \eqref{3.31} from  $
H^+ = \{ v \in H^1(\R) : v' \equiv 0 \,\,{\rm on}\ {\mathbb R}\setminus\Omega_0\},$ we see that the mapping $F_\e : H^+ \rightarrow \R$ given by
\begin{equation}
\label{decaye1}
F_\e(v) : =   \int_{\Omega_0} a_0 z'_\e v' - \lambda_\e \int_{\mathbb R} \rho z_\e v, \qquad v \in H^+,
\end{equation}
is linear and continuous, {\it i.e.} $F_\e$ belongs to $H^{-}$, the space of bounded linear functionals on $H^+$, with
\begin{equation}
\label{decaye3}
\lim_{\e \rightarrow 0}||F_\e||_{H^-} =0.
\end{equation}
In 
Appendix \ref{rem:limithom} below, we use standard arguments to demonstrate that there is a unitary map $\Psi \circ \mathcal{U}$ and an element $f_\e $ of the space
\begin{equation}
 \mathfrak{h}^-:= \bigl\{\text{$f : (0,2\pi) \rightarrow \ell^2$ measurable}: \big( {\lambda}_n(\theta) + 1\big)^{-1/2} f(\theta, n) \in L^2(0,2\pi;\ell^2)\bigr\},
\label{hfrakminus}
\end{equation}  
such that 
\begin{equation}
\label{decaye2}
\begin{aligned}
		_{H^-}\langle F_\e, v \rangle_{H^+} & 
		=  \sum_{n \in \N}\int_0^{2 \pi} f_\e(\theta,n) \overline{( \Psi \mathcal{U}) v (\theta,n) } \, {\rm d}\theta \quad \ \ \ \forall v \in H^+, \\
|| F_\e ||_{H^-} & = \sqrt{  \sum_{n \in \N}\int_0^{2 \pi} \dfrac{\bigl| f_\e(\theta,n)\bigr|^2}{{\lambda}_n(\theta) + 1}  \, {\rm d} \theta}.
	\end{aligned}
\end{equation}
(We recall that $\l_n(\theta)$ are the eigenvalues of the operator $A_\theta$, see Section \ref{sec:probform}). Now, by applying the transform $\Psi \mathcal{U}$ to \eqref{decaye1}, we find that 
$$
	_{H^-}\langle F_\e, v \rangle_{H^+} =   \sum_{n \in \N}\int_0^{2 \pi}\bigl( \lambda_n(\theta) - \lambda_\e\bigr)( \Psi \mathcal{U}) z_\e(\theta,n) \overline{( \Psi \mathcal{U})v(\theta,n) }  \, {\rm d}\theta.
$$
This equality, the formulae \eqref{decaye2} and the fact that $\Psi\mathcal{U}$ unitarily maps $H^{-}$ to $\mathfrak{h}^-$ implies that 
\[
f_\e({\theta, n}) = ( \lambda_n(\theta) - \lambda_\e)( \Psi\mathcal{U}) z_\e(\theta,n)
\] 
almost everywhere in $\theta,$ and 
$$
	|| F_\e ||_{H^{-1}}^2 =   \sum_{n \in \N}\int_0^{2 \pi} \dfrac{\bigl(\lambda_n(\theta) - \lambda_\e\bigr)^2}{{\lambda}_n(\theta) + 1} \bigl|  ( \Psi \mathcal{U}) z_\e(\theta,n)\bigr|^2 \, {\rm d} \theta.
$$
By assumption, one has $\lambda_\e \rightarrow \lambda_0 \notin {\bigcup\limits_\theta \sigma(A_\theta)} = \overline{\sum\limits_{n \in \N}\bigl[\min\limits_{\theta} \lambda_n(\theta), \max\limits_\theta \lambda_n(\theta)\bigr] }$, and therefore there exists a constant $c>0$ such that for sufficiently small $\ep$ the inequality $\bigl| \lambda_n(\theta) - \lambda_\e \bigr| > c$ holds for all $n\in \N$ and all $\theta \in [0,2\pi)$. Hence, the above equality and \eqref{decaye3} imply that
$$
\lim_{\e \rightarrow 0} \sum_{n \in \N}\int_0^{2 \pi}\bigl|( \Psi \mathcal{U}) z_\e(\theta,n)\bigr|^2 \, {\rm d}\theta\le c \lim_{\e \rightarrow 0}|| F_\e ||_{H^{-1}}^2 = 0.
$$
Finally, since $( \Psi \mathcal{U})z_\e = ( \Psi \mathcal{U} \mathcal R_\e) (\chi_{\e,\alpha}v_\e)$, and  $\Psi\mathcal{U}\mathcal R_\e$ is unitary, it follows that Claim 2 holds. \end{proof}

\section{Extreme localisation of defect eigenfunctions}\label{s:decay}
\label{sec:decay}

In Section \ref{sec:uppersemicont} we demonstrate that for eigenvalue sequences converging to a point in a gap in the limit spectrum ${\bigcup_\theta \sigma(A_\theta)}$, the corresponding eigenfunctions $u_\e$ converge to zero in $L^2$ outside the defect $D$, as $\e\to 0$ for  $\alpha\in(0,1).$ In this section, using the fact that one-dimensional problems admit an explicit form of solutions in terms of the fundamental system and employing standard techniques from the theory of ordinary differential equations, we provide a stronger statement on the rate of decay outside the defect. Namely, we show that the eigenfunctions $u_\e$ decay at an accelerated exponential rate  outside of the defect, which is Theorem \ref{mainthm}, Claim 3.

As in Section \ref{sec:uppersemicont}, we assume a sequence of eigenvalues $\lambda_\ep$ of $A^\e_D$ converges to  $\lambda_0 \in \mathbb{R} \backslash {\bigcup_\theta \sigma(A_\theta)}$ as $\e \rightarrow 0$, and consider the  corresponding sequence $u_\e$ of $L^2(\mathbb{R})$-normalised eigenfunctions, {\it i.e.}
\begin{equation*}
\int_\R a^\e_D u_\e' \varphi'  =  \l_\e \int_\R \rho^\e_D u_\e \varphi, \qquad \forall \varphi \in H^1(\R).
\end{equation*}

Recalling the unitary operator $\mathcal R_\e : L^2_{\rho_\e}(\R) \rightarrow L^2_\rho(\R)$ given by $\mathcal R_\e(f)(y) = \e^{1/2} f(\e y)$, we note that for all $z\in{\mathcal I}_\varepsilon$ (see (\ref{triangle})), the function $\tilde u_\e := \mathcal{R}u_\e$ solves 
\begin{gather}
-( a_0  \tilde u_\e')' = \l_\e \rho_0 \tilde u_\e \qquad{\rm on}\ Y_0+z, \label{8.95b} \\
-\ep^{-2}( a_1  \tilde u_\e')' = \l_\e \rho_1 \tilde u_\e \qquad {\rm on}\  Y_1+z  , \label{8.95c}
\end{gather}
 and satisfies the interface conditions
\begin{equation}\label{8.97}
\begin{aligned}
\tilde u_\e\vert_{Y_0+z}(z+h)  =\tilde u_\e \vert_{ Y_1+z  }(z+h), & \quad & (a_0 \tilde u_\e') \big((z+h)^-\big)  =\ep^{-2}(a_1 \tilde u_\e')\big((z+h)^+\big), \\[0.4em]
\tilde u_\e\vert_{Y_0+z+1}(z+1) =\tilde u_\e \vert_{ Y_1+z  }(z+1), & &(a_0 \tilde u_\e')\big((z+1)^+\big)   =\ep^{-2}(a_1 \tilde u_\e') \big((z+1)^-\big).
\end{aligned}
\end{equation}

There exist 
solutions $v^\e_1,v^\e_2$ to the equation
$
-( a_0  u')' = \l_\e \rho_0 u,$ 
on $Y_0$,
and solutions $w_1^\e, w_2^\e$ to the equation
 $
 -\ep^{-2}( a_1  u')' = \l_\e \rho_1 u,$ 
 on $Y_1$,
 such that
\begin{equation*}
\begin{aligned}
\left. \left(\begin{array}{cc}
v_1^{\e} & v_2^{\e}
\\
a_0 {v^\e_1}' & a_0\,{v^\e_2}'
\end{array}\right)\right\vert_{y=0}
=
\left(\begin{array}{cc}
1 & 0
\\
0 & 1
\end{array}\right),
 & \qquad & 
\left. \left(\begin{array}{cc}
w_1^{\e} & w_2^{\e}
\\
a_1 {w^\e_1}' & a_1\,{w^\e_2}'
\end{array}\right)\right\vert_{y=h}
=
\left(\begin{array}{cc}
1 & 0
\\
0 & 1
\end{array}\right).\end{aligned}
\end{equation*}
The solution  $\tilde u_\e$ to  (\ref{8.95b}), (\ref{8.95c}), $z\in{\mathcal I}_\varepsilon,$ admits the representation 
\begin{equation}\label{8.111}
\tilde u_\e(y) = \left\{\begin{array}{lcr}
a^\e_z v^\e_1(y-z)+ b^\e_z v^\e_2(y-z), & \  & y\in Y_0+z,
\\[0.5em]
c^\e_z w^\e_1(y-z)+ d^\e_z w^\e_2(y-z), & \ & y\in Y_1+z   .
\end{array}\right.
\end{equation}
For all $\varepsilon,$ the coefficients $a^\e_z, b^\e_z$ $c^\e_z$ and $d^\e_z,$ $z\in{\mathcal I}_\varepsilon,$ are related to each other by the conditions (\ref{8.97}), as follows:
$$
\begin{aligned}
c^\ep_z = a^\e_z v_1^\e(h) + b^\e_z v^\e_2(h), & \qquad & \ep^{-2}d^\ep_z =  a^\e_z (a_0{v_1^\e}')(h) + b^\e_z (a_0 {v^\e_2}')(h), \\[0.2em]
a^\ep_{z+1} = c^\e_z w_1^\e(1) + d^\e_z w^\e_2(1), & \qquad & \ep^{2}b^\ep_{z+1} = c^\e_z (a_1{w_1^\e}')(1) + d^\e_z (a_1 {w^\e_2}')(1).
\end{aligned}
$$
Eliminating $c^\e_z$ and $d^\e_z$ gives the iterative system
\begin{equation}
\left(\begin{array}{c}
a^\e_{z+1}
\\
b^\e_{z+1}
\end{array}\right)
=
M_\e
\left(\begin{array}{c}
a^\e_z
\\
b^\e_z
\end{array}\right),
\end{equation}
where the matrix $M_\e$ is given by
\begin{equation}
M_\e
=
\left(\begin{array}{cc}
v^\e_1(h) w^\e_1(1)  + \e^2 (a_0 {v^\e_1}')(h) w^\e_2(1) &  v^\e_2(h) w^\e_1(1) + \e^2 (a_0 {v^\e_2}')(h) w^\e_2(1)
\\[0.5em]
\e^{-2}  v^\e_1(h) (a_1{w^\e_1}')(1)  + (a_0 {v^\e_1}')(h) (a_1 {w^\e_2}')(1) & \e^{-2} v^\e_2(h) (a_1 {w^\e_1}')(1)  +(a_0{v^\e_2}')(h) (a_1 {w^\e_2}')(1)
\end{array}\right).
\end{equation}
It follows from the property that the modified Wronskian is constant,
\begin{equation*}
\det  \left(\begin{array}{cc}
v_1^{\e} & v_2^{\e}
\\
a_0 {v^\e_1}' & a_0\,{v^\e_2}'
\end{array}\right) 
\equiv 1, \qquad
\det  \left(\begin{array}{cc}
w_1^{\e} & w_2^{\e}
\\
a_1 {w^\e_1}' & a_1\,{w^\e_2}'
\end{array}\right)
\equiv 1,
\end{equation*}
that the characteristic polynomial of $M_\e$ is given by
\begin{equation}\label{8.117}
\begin{aligned}
\det (M_\e - \mu I) & = \mu^2 - \mu h_\e +1, \\[5pt]
h_\e = v^\e_1(h) w^\e_1(1) + \e^2 (a_0 {v^\e_1}')(h) w^\e_2(1) + &\e^{-2} v^\e_2(h)(a_1{w^\e_1}')(1)+(a_0 {v^\e_2}')(h) (a_1 {w^\e_2}')(1).
\end{aligned}
\end{equation}
 Recalling, from Section \ref{sec:transfer}, the fundamental solutions $v_1$, $v_2$ of ({\it cf.} (\ref{7.76}), (\ref{equivlim})) 
 \begin{equation*}
 - (a_0 u')' = \lambda_0 \rho_0 u  \quad \text{in $Y_0$}, 
 \end{equation*}
 satisfying
  \begin{equation*}
  \begin{aligned}
  \left(\begin{array}{cc} v_1(0) & v_2(0) \\[0.4em] (a_0 v_1')(0) & (a_0 v_2')(0)\end{array}\right)
  = \left(\begin{array}{cc} 1 &0\\[0.4em] 0 & 1\end{array}\right),
  \end{aligned}
  \end{equation*}
 we shall prove in the second half of this section the following property.
 \begin{lemma}\label{l6.1}
 The following convergence holds:
  \begin{equation}\label{toshow}
  \begin{aligned}
  \lim_{\e \rightarrow 0}h_\e=  v_1(h) + (a_0 v_2')(h) - \lambda_0 v_2(h)\int_{Y_1} \rho_1 .
  \end{aligned}
  \end{equation}	
 \end{lemma} 

 Assuming that (\ref{toshow}) holds, since $\l_0\in\mathbb{R} \backslash{\bigcup_\theta \sigma(A_\theta)},$ or equivalently (see Section \ref{sec:transfer}) $\lambda_0$ is such that ({\it cf.} (\ref{specineq}))
 \begin{equation*}
 \bigg|  v_1(h) + (a_0 v_2')(h) - \lambda_0v_2(h)\int_{Y_1} \rho_1 \bigg| > 2,
 \end{equation*}
for sufficiently small $\ep$ we find that $|h_\e|> 2$. 

As per the discussion in Section \ref{sec:formalasymp}, the roots $\mu^\e_1, \mu^\e_2$ of the matrix $M_\ep$ satisfy the identity $\mu^\e_1 \mu^\e_2 = 1$ and the nature of $\tilde u_\e$ away from the defect is determined by the coefficient $h_\e$.  
In particular, if $| h_\e | > 2$ then the roots $\mu_1^\e, \mu_2^\e$ are such that $|\mu_1^\e| <1$ and $|\mu_2^\e|>1$ and there exist linearly independent functions $v_{\rm g},v_{\rm d}$ on $\R \backslash\bigl(-\lfloor d_- \rfloor_\e, \lceil d_+\rceil_\e\bigr)$ that grow and decay respectively. In this case, for $u_\e$ be an element of $L^2(\R)$ it is necessary that $u_\e$ is proportional to the decaying solution  $v_{\rm d}$, which takes the form
$$
v_{\rm d}(x) = \left\{ \begin{array}{lcr}
\exp\Bigl(\frac{\ln | \mu_1^\e|}{\e}  \dist(x,D)\Bigr)\, p_1^\e(x/\e), &x\in [d_+,\infty),\\[0.5em]
\exp\Bigl(\frac{\ln | \mu_1^\e|}{\e}  \dist(x,D)\Bigr)\, p_2^\e(x/\e),  & x\in (-\infty, d_-],
\end{array}
\right.
$$
for some periodic (respectively, anti-periodic) functions $p_1^\e$, $p_2^\e$, when $h_\e >2$ (respectively, when $h_\e < 2)$. Therefore, for any $\nu$ satisfying $\nu < - \ln | \mu_1^\e | = \big| \ln | \mu_1^\e | \big|$ the product  $g_{\nu/\e}u_\e$ is in $L^2(\O)$, where $g_{\nu/\e}$ is defined by (\ref{2.4}). Then  the third claim of Theorem \ref{mainthm} follows by noticing  that by \eqref{toshow} $\mu_1^\e$ converges to $\mu_1$, the smallest root of 
$
\mu^2 - h \mu +1,$ where 
$$
h:= v_1(h) + (a_0 v_2')(h) - \lambda_0 v_2(h)\int_{Y_1} \rho_1 ,
$$
 as $\e \rightarrow 0$.
 
 It remains to prove the convergence \eqref{toshow}.

\smallskip

{\bf Proof of Lemma \ref{l6.1}.}

 The vector field
\begin{equation}
\label{evolution}
\eta^\e_j : = \left( \begin{matrix}
v^\e_j - v_j \\[0.4em]
a_0{v^\e_j}' - a_0{v_j}' 
\end{matrix}\right), \qquad j=1,2,
\end{equation}
solves the initial-value problem 
\begin{equation}\label{6.63}
\begin{aligned}
{\eta^\e_j}' = \Phi^\e \eta^\e_j + \Psi^\e_j \quad \text{in $Y_0$}, & \qquad &  \eta^\e_j(0) = 0,\ \ \ j=1,2,
\end{aligned}
\end{equation}
for the matrix $\Phi^\e$ and vector $\Psi^\e_j,$ $j=1,2,$ given by
$$
\begin{aligned}
\Phi^\e = \left( \begin{matrix}
0 & a^{-1}_0 \\[0.4em] - \lambda_\e \rho_0 & 0 
\end{matrix}\right), & \qquad & \Psi^\e_j =  \left( \begin{matrix}
0 \\[0.4em]  (\lambda_0 - \lambda_\e) \rho_0 v_j
\end{matrix}\right),\ \ \ \ j=1,2.
\end{aligned}
$$
Since $\l_\e \to \l_0$ the solutions to (\ref{6.63}) converge uniformly on $Y_0$ to the trivial solution of
\begin{equation*}
\eta ' = \Phi \eta \quad \mbox{ in } Y_0, \qquad \eta(0) = 0,
\end{equation*}
where $\Phi$ is the limit of $\Phi_\e$, as $\e\to 0$ (see {\it e.g.} \cite[Theorem 1.6.1]{Zettl}). Namely, we have
$$
\left| \eta^\e_j(y) \right| = | \eta^\e_j(y) - \eta(y)| \le C \left|\lambda_\e - \lambda_0\right|, \quad j=1,2,
$$ 
for some constant $C$ independent of $\e$. In particular, recalling \eqref{evolution}, it follows that
\begin{equation}
\label{lce3.0}
\begin{aligned}
\lim_{\e \rightarrow 0} v^\e_j(h) =  v_j(h),\quad & \quad \lim_{\e \rightarrow 0}(a_0 {v^\e_j}')(h) = (a_0v_j')(h), \qquad j=1,2.
\end{aligned}
\end{equation}

Similarly, it is easy to see that $w_j^\e$ and $a_1{w_j^\e}'$ converge uniformly on $Y_1$ to $w_j$ and $a_1 w_j'$, where $w_j,\,j=1,2$ are the solutions of $(a_1 w')' = 0$ satisfying 
\begin{equation*}
\begin{aligned}
\left(\begin{array}{cc}
w_1(h) & w_2(h)
\\[0.4em]
(a_1 w_1')(h) & (a_1 w_2')(h)
\end{array}\right)
=
\left(\begin{array}{cc}
1 & 0
\\[0.4em]
0 & 1
\end{array}\right).\end{aligned}
\end{equation*}
Since $w_1 \equiv 1$ and $a_1 w_2' \equiv 1$ on $Y_1$ we see that 
\begin{equation}\label{6.65}
 \left(\begin{array}{cc}
	w_1^{\e} & w_2^{\e}
	\\[0.4em]
	(a_1{w^\e_1}') & (a_1{w^\e_2})'
	\end{array}\right)  \to \left(\begin{array}{cc}
	1 & \int_h^y a_1^{-1}
	\\[0.4em]
	0 & 1
	\end{array}\right) \mbox{ uniformly on $Y_1$ as }\e\to 0.
\end{equation} 

Furthermore, by the fundamental theorem of calculus and the fact $-\e^{-2} ( a_1 {w^{\e}_{1}}')'=\l_\e \rho_1 w^\e_1$,  we have
$$
\begin{aligned}
\ep^{-2}(a_1 {w^\e_1})'(1) - \ep^{-2}(a_1 {w^\e_1})'(h)  = -\lambda_\e \int_{h}^{1} \rho_1 w^\e_1,
\end{aligned}
$$
and since
$$
\int_{h}^{1} \rho_1 w^\e_1 - w^\e_1(h) \int_{h}^{1} \rho_1 = \int_{h}^{1} \rho_1\bigl(w^\e_1 - w^\e_1(h)\bigr) = \int_{Y_1} \rho_1(y) \left( \int_h^y{w^\e_1}' \right) \, {\rm d} y,
$$
it follows that
\begin{align*}
\left| \ep^{-2}(a_1 {w^\e_1}')(1^-) - \ep^{-2}(a_1 {w^\e_1}')(h^+) + \lambda_\e w^\e_1(h)  \int_{h}^{1} \rho_1 \right|& =  \left| \lambda_\e\int_{Y_1} \rho_1(y) \left( \int_h^y{w^\e_1}' \right) \, {\rm d} y \right| \\
& \le |\lambda_\e | || \rho_1 ||_{L^\infty}||{w^\e_1}'||_{L^\infty},
\end{align*}
which together with \eqref{6.65} implies
\begin{equation*}
\lim_{\e \rightarrow 0}\left| \ep^{-2}(a_1 {w^\e_1}')(1) - \ep^{-2}(a_1 {w^\e_1}')(h) + w^\e_1(h) \lambda_\e \int_{h}^{1} \rho_1 \right| = 0.
\end{equation*}
Therefore
\begin{equation}
\label{lce4}
\lim_{\e \rightarrow 0}\ep^{-2}\bigl(a_1{w^\e_1}'\bigr)(1)=-\lambda_0 \int_{Y_1} \rho_1.
\end{equation}
Finally, assertions \eqref{lce3.0}, \eqref{6.65} and \eqref{lce4} imply \eqref{toshow}, as required. 

%
\section{Resolvent estimates for the problem without defect}
\label{sec:resolvent}
In this section we study the behaviour of the unperturbed periodic operator $A^\e$ in the operator norm as $\e \rightarrow 0$. In particular, we construct a full asymptotic expansions for the resolvent of $A^\e$ using  a version of the asymptotic framework developed in \cite{ChCoARMA}, see Theorem \ref{lem:asymp} below. This directly implies the order-sharp operator norm resolvent convergence estimate, uniform in $\theta$, formulated in Theorem \ref{estlemma}. The latter, in turn, implies the uniform in $\theta$ convergence, as $\e\to 0$, of the spectral band functions $\l_n^\e(\theta)$ to $\l_n(\theta)$, $n\in \N$, which is also order-sharp.

Recall the operator $A^\e$ in $L^2_{\rho^\e}(\R)$ associated with the bilinear form
\begin{equation*}
\begin{aligned}
\beta^\e (u,v) =   \int_{\Omega^\e_1 } a_1 (\tfrac{\cdot}{\e}) u'\overline{v'}+  \int_{\Omega^\e_0 } \ep^2 a_0 (\tfrac{\cdot}{\e})  u'\overline{v'}, \qquad u,v \in H^1(\RR).
\end{aligned}
\end{equation*}
By a scaled version of the Floquet-Bloch transform\footnote{See Appendix A below for further information on the Floquet-Bloch transform.} which is given as the continuous extension of the following action on {\it e.g.} continuous functions with compact support
\begin{equation}
\label{rsgelf}
({\mathcal U}_\varepsilon f)(\theta, y)=\sqrt{\frac{\varepsilon}{2\pi}}\sum_{z\in{\mathbb Z}}f\bigl(\varepsilon(y-z)\bigr){\rm e}^{{\rm i}\theta z}, \qquad y \in Y,\ \theta\in[0,2\pi),
\end{equation}
we see that ${\mathcal{U}}_\e$ unitarily maps $L^2_{\rho_\e}(\R)$ to the Bochner space $L^2\bigl(0,2\pi;L^2_\rho(Y)\bigr)$ and 
${\mathcal{U}}_\e A^\e f(\theta,\cdot) = A^\e_\theta {\mathcal{U}}_\e  f(\theta,\cdot) $. 
Here, $A^\e_\theta$ is the operator defined in $L^2_\rho(Y)$ and associated with the form
$$
\beta^\e_1 (u,v) : =  \int_{Y_0} a_0 u'\overline{v'} + \ep^{-2}\int_{Y_1} a_1 u'\overline{v'}, \qquad u,v \in H^1_\theta(Y).
$$
We recall that $H^1_\theta(Y)$ is the complex Hilbert space of $H^1(Y)$-functions that are $\theta$-quasiperiodic. We equip the space 
$H^1_\theta(Y)$ with the graph norm
\begin{equation}
\label{graphnorm}||| u ||| : = \sqrt{\int_{Y_0} a_0 |u'|^2 + \int_{Y_1} a_1 |u'|^2 + \int_Y \rho |u|^2},
\end{equation}
and consider the subspace
$$
V_\theta : = \bigl\{v \in H^1_\theta(Y): v' \equiv 0 \text{ in $Y_1$} \bigr\}
$$
and its orthogonal complement $V^\perp_\theta$ in $H^1_\theta$ with respect to the  inner product associated with $||| \cdot |||$. The following result, established in \cite{ChCoGu}, is of fundamental importance in studying the asymptotics of $A^\e$, equivalently $A^\e_\theta$.
\begin{lemma}
	\label{lem:uniformtheta}
	There exists a constant $C_P>0$, independent of $\theta$, such that
	\begin{equation}\label{7.69}
		||| P^\perp_\theta u ||| \le C_P || \sqrt{a_1} u' ||_{L^2(Y_1)}, \qquad \forall u \in H^1_\theta(Y),
	\end{equation}
	where $P^\perp_\theta$ is the orthogonal projection of $H^1_\theta(Y)$ onto $V^\perp_\theta$.
\end{lemma}
For $\theta \in [0,2\pi)$ and all $f \in L^2_\rho(Y),$ we  consider the resolvent problem
\beq
\begin{aligned}-\big(( \ep^{-2}a_1 + a_0){u^\e_\theta}' \big)'
	+\rho u_\theta^\varepsilon=\rho f& \qquad & {\rm on}\ (0,1).
\end{aligned}
\eeq{eqn}
We look for an asymptotic expansion of $u_\theta^\e$ in the form
\beq
u_\theta^\ep=\sum_{n=0}^\infty\varepsilon^{2n}u_\theta^{(2n)},\ \ \ \ \ \ \ \ u_\theta^{(2n)}\in H^1_\theta(Y)\ \ \ \forall n \in \N.
\eeq{asympexp} 
The following result holds.
\begin{theorem}
	\label{lem:asymp}
	For each $\theta \in [0,2\pi)$ and $f \in L^2_\rho(Y),$ consider the unique solution $u^{(0)}_\theta \in V_\theta$ to the problem
	$$
	\int_{Y_0}   a_0(u_\theta^{(0)})'\overline{\varphi'}  + \int_Y \rho	u_\theta^{(0)} \overline{\varphi} = \int_Y \rho f\overline{ \varphi} \qquad \forall\varphi\in V_\theta,
	$$
	and for all $n\in{\mathbb N}$ consider the unique solution $u^{(2n)}_\theta \in V^\perp_\theta$ to 
	\begin{equation*}
	- \Big(a_1\bigl(u^{(2n)}_\theta\bigr)' \Big)'= \Big(a_0\bigl(u^{(2(n-1))}_\theta\bigr)' \Big)' -\rho u_\theta^{(2(n-1))}+\delta_{1n}\rho f, 
	\end{equation*}
	where $\delta_{1n}$ is the Kronecker delta function. Then, for each $N \in \N$ the sum
	$$
	U^{(N)}_\theta : = \sum_{n=0}^N \ep^{2n} u^{(2n)}_\theta
	$$
	approximates the solution $u^\e_\theta$ to \eqref{eqn} in the following sense:
	\[
	||| u_\theta^\varepsilon-U^{(N)}_\theta |||\le C_P^{2(N+1)}\varepsilon^{2(N+1)}\bigl\Vert f \bigr\Vert_{L^2_\rho(Y)}.
	\]
\end{theorem}

\begin{remark}
\label{rem:spectralconvergence}
By an application of the min-max principle, Theorem \ref{lem:asymp} implies that the $n$-th eigenvalue $\lambda^\ep_n(\theta)$ of the operator $A^\e_\theta$ is $\ep^2$-close, uniformly in $\theta,$ to the $n$-th eigenvalue $\lambda_n(\theta)$ of $A_\theta$, {\it i.e.} for each $n \in \N$   there exists a constant $c_n >0$ such that
$$
\bigl| \lambda^\ep_{ n }(\theta) - \lambda_{n }(\theta)\bigr| \le c_n \ep^2 \qquad \forall \theta \in [0,2\pi).
$$
In particular, this indirectly implies, since $\lambda_n$ is the uniform limit of continuous functions, that $\lambda_n$ is continuous in $\theta$. A direct proof of this fact can be arrived at by the definition of the operators $A_\theta$ and the continuity properties (in the Hausdorff sense) of their domains $D(A_\theta)$, see \cite[Appendix B]{ChCoGu}.
\end{remark}

\begin{proof}

Substituting (\ref{asympexp}) into (\ref{eqn}) and equating powers of $\varepsilon$ yields a system of recurrence relations for the functions $u^{(2n)}_\theta,$ $n\in \N.$ The first equation in this system, which corresponds to $\varepsilon^{-2},$ is
\begin{equation}
\label{u0eq}
\begin{aligned}
-\Big(a_1\bigl({u^{(0)}_\theta}\bigr)' \Big)' = 0& \qquad& \quad{\rm on}\ \ (0,1),
\end{aligned}
\end{equation}
which implies that $u^{(0)}_\theta \in V_\theta = \bigl\{ v \in H^1_\theta(Y) : v' \equiv 0\ {\rm on}\ Y_1\bigr\}$ (recall that $a_1 \equiv 0$ on $Y_0$). The remaining equations, obtained by considering the terms of order $\varepsilon^{2j},$ $j=0,1,2, ...$ are
\begin{equation}
\begin{aligned}
- \Big( a_1\bigl(u^{(2n)}_\theta\bigr)' \Big)'= \Big( a_0\bigl(u^{(2(n-1))}_\theta\bigr)' \Big)' -\rho u_\theta^{(2(n-1))}+\delta_{1n}\rho f, & \qquad& \quad{\rm on}\  (0,1),\ \ \ \ n\in{\mathbb N},
\end{aligned}
\label{u1eq}
\end{equation}
where, as before, $\delta_{in}$ denotes the Kronecker delta function. The existence of solutions to  differential equations with degenerate coefficients such as \eqref{u1eq}  was first studied in \cite{IkVs:pdhom} for the case $\theta=0$, and it was shown therein that existence is guaranteed by inequalities of the type (\ref{7.69}). By following this general framework, and it can be readily shown that \eqref{7.69} implies the following result.
\begin{lemma}
	\label{lemka}
	For a given $F \in H^{-1}_\theta(Y)$, the dual space of $H^1_\theta(Y)$, there exist (infinitely many) solutions $u$ to the problem
	$$
	\int_{Y_1} a_1 u' \overline{\varphi'} = _{H^{-1}_\theta(Y)}\langle F, \varphi \rangle_{H^1_\theta(Y)} \qquad \forall \varphi \in H^1_\theta(Y),
	$$ 
	if and only if $F$ satisfies the condition 
	\[
	_{H^{-1}_\theta(Y)}\langle F, v \rangle_{H^1_\theta(Y)} = 0\ \ \ \ \forall v \in V_\theta.
	\]
	Such solutions are unique in $V^\perp_\theta$, {\it i.e.} for any two solutions $u_1$, $u_2$ one has $u_1-u_2 \in V_\theta$. 
\end{lemma}
Consequently,  the system (\ref{u1eq}) is solvable if and only if the conditions
\beq
\int_{Y_0}   (a_0 {u_\theta^{(2n)}})'   \overline{\varphi'}  + \int_Y\rho	u_\theta^{(2n)} \overline{\varphi} = \delta_{0n} \int_Y\rho f\overline{ \varphi} \ \ \ \ \forall\varphi\in V_\theta, \qquad n \in \N,
\eeq{solvcond}
hold. The equation for $n =0$ uniquely determines $u^{(0)}_\theta$ and for $n \ge 1$, due to the choice (\ref{graphnorm}) of the norm on $H^1_\theta(Y)$,  demonstrates that $u^{(2n)}_\theta \in V^\perp_\theta$. 
Substituting $\varphi=u_\theta^{(0)}$ into the identity (\ref{solvcond}) for $n=0$,  recalling \eqref{graphnorm}, the fact that $a_1({u^{(0)}_\theta})' \equiv 0$  and using the Cauchy-Schwarz inequality, we obtain
\[
||| u_\theta^{(0)}|||^2\le
\Vert f\Vert_{L^2_\rho(Y)}\bigl\Vert u_\theta^{(0)}\bigr\Vert_{L^2_\rho(Y)}.
\] 
Hence, $u^{(0)}_\theta$  satisfies the bound
\begin{equation}
\label{u0bound}
\begin{aligned}
|||u_\theta^{(0)}|||\le
\Vert f\Vert_{L^2_\rho(Y)} & \qquad& \forall\theta\in[0,2\pi).
\end{aligned}
\end{equation}
By Lemmas \ref{lem:uniformtheta} and \ref{lemka},  the solution $u^{(2n)}_\theta\in V_\theta^\perp$ to \eqref{u1eq} is unique and
\begin{equation}
||| u^{(2n)}_\theta ||| \le C_P\bigl\Vert \sqrt{a_1} u^{(2n)}_\theta \bigr\Vert_{L^2(Y_1)}. 
\label{Poincare}
\end{equation}
Equations (\ref{u1eq}) and the orthogonality of $V_\theta$ and $V^\perp_\theta$ with respect to the norm (\ref{graphnorm}), in particular, the orthogonality of $u_\theta^{(0)}$ and  $u_\theta^{(2)}$, imply that
$$
\int_{Y_1} a_1\Bigl|\bigl(u^{(2n)}_\theta\bigr)'\Bigr|^2 = \delta_{1n} \int_Y \rho f \overline{ u^{(2n)}_\theta} - (1-\delta_{1n}) \left( \int_{Y_0} a_0\bigl(u^{(2(n-1))}_\theta\bigr)'\overline{\bigl(u^{(2n)}_\theta\bigr)'} + \int_Y \rho u^{(2(n-1))}_\theta \overline{u^{(2n)}_\theta}\right), \quad n \ge 1,
$$
and (\ref{Poincare}) yields
\begin{equation*}
\begin{aligned}
||| u^{(2)}_\theta|||\le C^2_P\bigl\Vert  f\bigr\Vert_{L^2_\rho(Y_1)}, & \qquad & |||u^{(2n)}_\theta|||\le C^2_P|||  u^{(2(n-1))}|||, \quad n \ge 2. 
\end{aligned}
\end{equation*}
By iterating the above inequalities we establish that
\begin{equation}
\begin{aligned}
||| u^{(2n)}_\theta||| \le C^{2n}_P \bigl\Vert  f\bigr\Vert_{L^2_\rho(Y_1)},  \qquad n \ge 1. 
\end{aligned}
\label{u1estimate}
\end{equation}

Having identified each term in the expansion, for each $n \in \N$ we define the remainder $R_\theta^\ep$ (dropping the index $N$ for brevity), according to the formula
\begin{equation}
u_\theta^\varepsilon= \sum_{n=0}^{N} \ep^{2n} u^{(2n)}_\theta +\ep^{2N} R_\theta^\varepsilon,
\label{remainderdef}
\end{equation}
and find, via \eqref{u0eq} and \eqref{u1eq}, that $R_\theta^\varepsilon \in H^1_\theta(Y)$ solves the problem
\[
\begin{aligned}
- \big( (\ep^{-2} a_1+ a_0)(R^\varepsilon_\theta)'\big)' +\rho R_\theta^\varepsilon= \delta_{0N} \rho f + \big(a_0 (u_\theta^{(2N)})'\,\big)' -\rho u_\theta^{(2N)}& \qquad&{\rm on}\ (0,1),
\end{aligned}
\]
that is
\begin{multline*}
 \int_{Y_1}\ep^{-2} a_1 (R_\theta^\varepsilon)' \overline{v'}  +\int_{Y_0}  a_0 (R_\theta^\varepsilon)' \overline{v'} 
+\int_Y \rho R^\varepsilon_\theta  \overline{v} = \delta_{0N}\int_Y \rho f \overline{v} -\int_{Y_0} a_0 (u_\theta^{(2N)})'\,
 \overline{v'} -\int_Y\rho  u_\theta^{(2N)} \overline{v} , \\ \qquad \forall v \in H^1_\theta(Y).
\end{multline*}
Setting $v  \in V_\theta$, recalling the norm \eqref{graphnorm} and \eqref{solvcond}, demonstrates that $R^\ep_\theta \in V^\perp_\theta$. Additionally, setting $v = R^\ep_\theta$ above implies that 
\[
\ep^{-2}\int_{Y_1} a_1 \bigl\vert(R_\theta^\varepsilon)'\bigr\vert^2
\le\delta_{0N}\int_Y \rho f \overline{R_\theta^\varepsilon}-\int_{Y_0} a_0 (u_\theta^{(2N)})'\,
\overline{(R_\theta^\varepsilon)'}-\int_Y\rho  u_\theta^{(2N)}\overline{R_\theta^\varepsilon},
\]
and inequalities \eqref{7.69}, \eqref{u1estimate}, along with another application of the Cauchy-Schwarz inequality yields
\begin{equation*}
||| R_\theta^\varepsilon|||\le C^{2(N+1)}_P \ep^2\Vert f\Vert_{L^2_\rho(Y)}.
\end{equation*}
Finally, by combining this inequality with (\ref{remainderdef}) we deduce that
\[
||| u_\theta^\varepsilon-\sum_{n=0}^N \ep^{2n} u^{(2n)}_\theta|||\le  C^{2(N+1)}_P \varepsilon^{2(N+1)}\Vert f\Vert_{L^2_\rho(Y)},
\]
as required.

\end{proof}

\appendix
\section{Appendix: Norm-resolvent convergence of $A^\ep$ and the limit operator $A^0$}	\label{rem:limoperator}	

	We consider the space $H$ to be the closure in $L^2_\rho(\R)$ of ({\it cf.} the end of Section \ref{sec:uppersemicont})
	$$
	H^+=\bigl\{v\in H^1(\R): v' \equiv 0\ {\rm on}\ \Omega_1: = \bigcup_{z \in \Z}(Y_1+z)\bigr\}.
	$$
	Both $H$ and $H^+$ are Hilbert spaces when equipped with the inner products inherited from $L^2_\rho(\R)$ and $H^1(\R)$ respectively, and clearly $H^+$ is densely defined in $H$ with continuous embedding (recall $\rho$ is taken to be uniformly positive and bounded). The norm of $H^+$, which is the standard $H^1$-norm, is equivalent to the graph norm
	\begin{equation}
	|| \cdot ||_{H^+} : = \left( || \cdot ||^2_{L^2_\rho(\R)} + \beta^0(\cdot,\cdot) \right)^{1/2},
	\label{graph_norm}
	\end{equation}
	where $\beta^0$ is the bilinear form 
	$$
	\beta^0(u,v) := \int_{\Omega_0} a_0 u'\overline{v'}, \qquad u,v \in H^+.
	$$
	We shall henceforth consider $H^+$ to be equipped with the graph norm (\ref{graph_norm}), and denote by $H^{-}$ the dual space consisting of bounded linear functionals on $H^+$. As $\beta^0$ is a non-negative closed symmetric quadratic form it  generates a densely defined non-negative self-adjoint linear operator $A^0$. The domain $D(A^0)$ is the dense subset of $H^+$ consisting of the  solutions to the problem:  for each $f \in H$ we consider $u \in H^+$ the unique solution to the problem
	$$
	\beta^0(u,v) + \int_{\R} \rho u \overline{v} = \int_{\R} \rho f\overline{v} \qquad \forall v \in H^+,
	$$
	and set $A^0u = f -u$ for $u\in D(A^0)$. The operator $A^0$ is unitarily equivalent to the fibre integral operator $\int_{\theta} A_\theta$, {\it cf.} Remark \ref{rem:2.1}, and the unitary map is given by the continuous extension of the  Floquet-Bloch transform  $\mathcal{U}$, cf. \cite[Section 2.2]{Ku} which acts on smooth functions $f$ with compact support as 
	\begin{equation*}
	\begin{aligned}
	\mathcal Uf(\theta,y): = \frac{1}{\sqrt{2\pi}}\sum_{z\in \Z} f(y-z)e^{i  \theta z}, & \qquad& \text{$\theta \in [0,2\pi), \ \ y \in Y$}.
	\end{aligned}
	\end{equation*}
	Indeed, $\mathcal{U}$ is well-known to be a  unitary operator between $L^2_\rho(\R)$ and the Bochner space $L^2\bigl(0,2\pi;L^2_\rho(Y)\bigr)$ and it is straightforward to see that
	$$
	\mathcal{U} A^0 f(\theta; \cdot) = A_\theta\,\mathcal{U}f(\theta; \cdot),\ \ \ \ \forall f\in L^2_\rho(\R),\ \theta\in[0, 2\pi).
	$$
	Furthermore, it is clear that $\mathcal{U}$ unitarily maps $H^+$ to the space $L^2(0,2\pi; V_\theta)$ (we recall that $V_\theta =\bigl\{ v \in H^1_\theta(Y) : v' \equiv 0 \text{ on $Y_1$}\bigr\}$).  It is easy to verify that  the spectrum of $A^0$ coincides with the union of the spectra of $A_\theta$ over all 
	$\theta\in[0,2\pi)$, {\it i.e.}
	$$
	\sigma(A^0) = {\bigcup\limits_\theta \sigma(A_\theta)} = {\bigcup_{n \in \N}\bigl[\min\limits_{\theta} \lambda_n(\theta), \max\limits_\theta \lambda_n(\theta)\bigr] }.
	$$
	Theorem \ref{estlemma} implies in particular that $A^\ep$ converges at the rate $\ep^2$ in the norm-resolvent sense to $A^0$, {\it i.e.} there exists a constant $C>0$ such that  
	$$
	\bigl\Vert \mathcal{R}_\ep( A^\e  + 1)^{-1}\mathcal{R}_\ep^{-1} - (A^0+1)^{-1}\bigr\Vert_{L^2_\rho(\R) \rightarrow L^2_\rho(\R)} \le C \ep^2.
	$$
	for all $\varepsilon\in(0,1).$ As before, $\mathcal R_\e : L^2_{\rho^\e}(\R) \rightarrow L^2_\rho(\R)$ is the unitary transformation $\mathcal R_\e(f)(y) = \e^{1/2} f(\e y)$.
	
\section{Appendix: Spectral decomposition of $A^0$}	\label{rem:limithom}

	As the operator $A^0$ is self-adjoint, it has a spectral decomposition and we shall now characterise the space $H^+$ and its dual $H^-$ in terms of a realisation of this spectral decomposition. For each $\theta$, the self-adjoint operator $A_\theta$  has compact resolvent and for each of its eigenvalues $\lambda_n(\theta)$, $n \in \N$, we denote by $\psi_n(\theta;.)$ the corresponding $L^2_\rho(Y)$-normalised eigenfunction. Then the mapping $\Psi$ given by 
	$$
	\begin{aligned}
	\Psi f (\theta ; \cdot) = \{ c_n (\theta) \}_{n \in \N}, & \qquad& c_n(\theta):= \int_Y \rho(y) f(\theta, y) \overline{\psi_n(\theta;y)} \, {\rm d}y, 
	\end{aligned}
	$$ 
	unitarily maps $L^2\bigl(0,2\pi; L^2_\rho(Y)\bigr)$ to $\mathfrak{h} : = L^2(0,2\pi; \ell^2)$ so that
	$$
	\Psi \big( \mathcal{U} A^0 f \big)(\theta, n ) =  \lambda_n (\theta) \Psi \big(  \mathcal{U}f \big)(\theta,n),
	$$
	where for $u \in \mathfrak{h}$, we denote by $u(\theta,n)$ is the $n$-th element of the sequence $u(\theta)$.    It is easy to verify that $\Psi \circ \mathcal{U}$ unitarily maps $H^+$ to 
	$$
	\mathfrak{h}^+ := \bigl\{u(\theta,n) \in \mathfrak{h} :  \big( {\lambda}_n(\theta) + 1\big)^{1/2}u(\theta, n) \in \mathfrak{h}\bigr\}.
	$$
	By standard duality arguments, see for example \cite[Chapter 1, Section 6.2]{LiMa}, we show that $\Psi \circ \mathcal{U}$ unitarily maps $H^-$, the dual space of bounded linear functionals on $H^+$, to ({\it cf.} (\ref{hfrakminus}))
	$$
	\mathfrak{h}^- := \bigl\{\text{$f: (0,2\pi) \rightarrow \ell^2$ measurable}: \big( {\lambda}_n(\theta) + 1\big)^{-1/2}f(\theta, n) \in \mathfrak{h}\bigr\},
	$$
	in the sense that $F \in H^{-}$ if and only if there exists $f \in \mathfrak{h}^{-}$ such that ({\it cf.} (\ref{decaye2}))
	$$
	\begin{aligned}
	_{H^-}\langle F, v \rangle_{H^+} =  \sum_{n \in \N}\int_0^{2 \pi} f(\theta,n) \overline{ \big( \Psi\mathcal{U} \big) v (\theta,n) } \, {\rm d}\theta & \qquad & \forall v \in H^+,
	\end{aligned}
	$$
	and we have
	$$
	|| F ||_{H^{-1}} = \sqrt{  \sum_{n \in \N}\int_0^{2 \pi} \frac{\bigl| f(\theta,n)\bigr|^2}{{\lambda}_n(\theta) + 1}\, {\rm d} \theta}.
	$$

\section*{Acknowledgments}
KC and SC are grateful for the financial support of
the Engineering and Physical Sciences Research Council: Grant EP/L018802/2 ``Mathematical foundations of metamaterials: homogenisation, dissipation and operator theory'' for KC, and Grant EP/M017281/1 ``Operator asymptotics, a new approach to length-scale interactions in metamaterials" for SC.



\medskip






\begin{thebibliography}{9}


\bibitem{Agmon} Agmon, Sh. {\it Lectures on Exponential Decay of Solutions of Second-Order Elliptic Equations: Bounds on Eigenfunctions of N-body Schr\"{o}dinger Operators} (Mathematical Notes {\bf 29}, Princeton University Press (1982) 

\bibitem{AADH} Alama, S., Avellaneda, M., Deift, P. A., Hempel, R. On the existence of eigenvalues of a divergence-form operator $A+\lambda B$ in a gap of $\sigma(A)$. Asymptotic Anal., {\bf 8}(4) 311--344 (1994)

\bibitem{Birman1961} Birman, M. Sh. On the spectrum of singular boundary-value problems (Russian). {\it Math. Sb.} {\bf 55 (97)}, 125--174 (1961). English transl. in Eleven Papers in Analysis, AMS Transl. 2 {\bf 53}, AMS, Providence, RI, 23--60 (1966)

\bibitem{Birman_Solomjak} Birman, M. S., Solomjak, M. Z. {\it Spectral Theory of Self-Adjoint Operators in Hilbert Space,} Springer Netherlands (1987).

\bibitem{Cherdantsev1} Cherdantsev, M. Asymptotic analysis of some spectral problems in high contrast homogenisation and in thin domains. {\it PhD Thesis}, University of Bath (2008)

\bibitem{Cherdantsev2} Cherdantsev, M. Spectral convergence for high-contrast elliptic periodic problems with a defect via homogenization. {\it Mathematika} {\bf 55} (1--2), 29--57, 2009.

\bibitem{ChCoARMA} Cherednichenko, K. D., Cooper, S. Resolvent estimates for high-contrast elliptic problems with periodic coefficients. {\it Archive for Rational Mechanics and Analysis} {\bf 219} (3), 1061--1086 (2016).

\bibitem{ChCoGu}
Cherednichenko, K. D., Cooper, S., Guenneau, S. Spectral analysis of one-dimensional high-contrast elliptic problems with periodic coefficients. {\it Multiscale Modeling and Simulation} {\bf 13} (1) 72--98 (2015).


\bibitem{FiKl}
Figotin, A., Klein, A. Localized classical waves created by defects, {\it J. Statist. Phys.} {\bf 86}(1--2) 165--177 (1997).

\bibitem{IkVs:pdhom} Kamotski, I. V., Smyshlyaev, V. P.: Two-scale homogenization for a class of partially
degenerating PDE systems. {\it arXiv:1309.4579} (2011).

\bibitem{JKO} Jikov, V. V., Kozlov, S. M., Oleinik, O. A. {\it Homogenization of Differential Operators and Integral Functionals,} Springer, Berlin (1994).

\bibitem{KamSm1} Kamotski, I. V., Smyshlyaev, V. P. Localised modes due to defects in high contrast
periodic media via homogenisation. {\it Bath Institute for Complex Systems, preprint 3/06} (2011)

\bibitem{Knight} Knight, J. C., Photonic crystal fibres. {\it Nature} {\bf 424} 847--851 (2003).

\bibitem{Krejcirik}
Krej\u ci\u r\'ik, D.,  Lu, Z. Location of the essential spectrum in curved quantum layers. \textit{J. Math. Phys.} {\bf 5} (2014).

\bibitem{Ku} Kuchment, P.  Floquet Theory for Partial Differential Equations. {\it Birkhauser Verlag, Basel,} (1993).

\bibitem{LiMa} Lions, J. L., Magenes, E. {\it Non-homogeneous Boundary Value Problems, Vol. 1,} Springer, Berlin, Heidelberg (1972).

\bibitem{Russell} Russell, P. St J. Photonic crystal fibers. {\it Science} {\bf 299} 385--362 (2003)

\bibitem{Zettl}
Anton Zettl, {\it Sturm-Liouville Theory (Mathematical Surveys and Monographs volume 121),} The American Mathematical Society, 2005.

\bibitem{Zhikov2000}
Zhikov, V. V. On an extension of the method of two-scale convergence and its applications, {\it Sbornik: Mathematics} {\bf 191}(7), 973--1014 (2000)

\bibitem{VL}
Vishik, M. I., Lyusternik, L. A. Regular degeneration and boundary layer for linear differential equations with small parameter. (Russian) {\it Uspekhi Mat. Nauk (N. S.)} (1957).

\bibitem{Weyl}
Weyl, H. \"{U}ber beschr\"{a}nkte quadratische Formen, deren Differenz vollstetig ist, {\it Rend. Circ. Mat. Palermo} {\bf 27}, 373--392 (1909)

\end{thebibliography}
\end{document}